\documentclass[onefignum,onetabnum]{siamart190516}
\usepackage{amsmath,amssymb,amsxtra}
\usepackage{mathrsfs}
\usepackage{float,verbatim}
\usepackage{threeparttable,booktabs}
\usepackage{graphicx}
\usepackage{epstopdf}
\usepackage{mathtools}
\usepackage{subfig}
\usepackage{algorithm, algpseudocode}
\graphicspath{{./pics/}}

\usepackage{array}

\newtheorem{example}{Example}[section]

\usepackage{graphicx}
\usepackage{dsfont}
\usepackage{epstopdf}
\graphicspath{{./pics/}}
\usepackage{booktabs}
\usepackage{threeparttable}
\usepackage{verbatim}
\newcommand\rmd {\,{\rm d}}
\newcommand{\bel}{\begin{equation} \label}
\newcommand{\ee}{\end{equation}}
\def\beq{\begin{equation}}
\def\eeq{\end{equation}}
\newcommand{\bea}{\begin{eqnarray}}
\newcommand{\eea}{\end{eqnarray}}
\newcommand{\beas}{\begin{eqnarray*}}
\newcommand{\eeas}{\end{eqnarray*}}

\allowdisplaybreaks

\numberwithin{equation}{section}

\allowdisplaybreaks

\def\phi {\varphi}

\allowdisplaybreaks

\allowdisplaybreaks

\title{An Iterative Direct Sampling Method for Reconstructing Moving Inhomogeneities in Parabolic Problems\thanks{The work of B. Jin is supported by Hong Kong RGC General Research Fund (14306423 and 14306824) and ANR / Hong Kong RGC Joint Research Scheme (A-CUHK402/24) and a start-up fund from The Chinese University of Hong Kong. The work of J. Zou was substantially supported by the Hong Kong RGC General Research Fund (projects 14310324, 14306623 and 14306921) and NSFC / Hong Kong RGC Joint Research Scheme 2022/23 (project N\_CUHK465/22).}}

\author{Bangti Jin\thanks{Department of Mathematics, The Chinese University of Hong Kong, Shatin, N.T., Hong Kong (email: \texttt{b.jin@cuhk.edu.hk, fengruwang@cuhk.edu.hk, zou@math.cuhk.edu.hk})}\and
Fengru Wang\footnotemark[2] \and
Jun Zou\footnotemark[2]
}

\date{\today}

\begin{document}

\maketitle

\begin{abstract}
We propose in this work a novel iterative direct sampling method for imaging moving inhomogeneities in parabolic problems using boundary measurements.
It can efficiently identify the locations and shapes of moving inhomogeneities when very limited data are available, even with only one pair of lateral Cauchy data, and enjoys remarkable numerical stability for noisy data and over an extended time horizon.
The method is formulated in an abstract framework, and is applicable to linear and nonlinear parabolic problems, including linear, nonlinear, and mixed-type inhomogeneities.
Numerical experiments across diverse scenarios show its effectiveness and robustness against the data noise.
\end{abstract}



\section{Introduction}

The estimation of inhomogeneities and the identification of coefficients in parabolic partial differential equations (PDEs) arise in a broad spectrum of applications.
This task arises in time-resolved diffuse optical tomography for imaging biological tissues \cite{Zhang2001}, material defect characterization \cite{Krings2012,Nair2017}, and thermal imaging \cite{Ammari2017} etc.
Thus there is an urgent need to develop relevant mathematical theory and computational techniques. There have been breakthroughs in the theoretical aspect of parabolic inverse problems, and various uniqueness results have been established.
For example, the lateral Dirichlet-to-Neumann (DtN) map uniquely determines the time-dependent linear potential or thermal conductivity \cite{Rozhkov1997,Isakov2017}, and a nonlinear potential term $a(x,u)$ in a semilinear parabolic equation \cite{Isakov1993}. See the monographs \cite{Cannon:1984,Isakov2017} for in-depth discussions on the theory of parabolic inverse problems.

Numerically, the problem is ill-posed in the sense that it lacks good stability with respect to the data perturbation, and thus challenging to solve.
The task becomes particularly challenging when only limited lateral Cauchy data are available, for which the inverse problems are severely ill-posed.
This is characterized by the high sensitivity of the solutions with respect to the noise in the data, which greatly complicates the development of robust and efficient reconstruction methods.
Traditionally, this class of inverse problems can be solved using variational regularization \cite{EnglHankeNeubauer:1996,ItoJin:2014} or iterative regularization \cite{Kaltenbacher2008}.
These approaches are general-purposed, and despite their great success, still face significant computational challenges.
For example, variational regularization requires tedious parameter tuning, repeated solutions of the resulting optimization problem, and a good initial guess in order to guarantee the convergence of the optimizer to the desired solution.
The convexification method due to Klibanov \cite{Klibanov:1997} provides globally convergent approximations with explicit error estimates, thereby circumventing the issue of local minima in the conventional regularization framework \cite{BeilinaKlibanov:2008,WenlongZhang:2020}.
Nonetheless, it remains restricted to specific measurement configurations, and its application to limited lateral Cauchy data is still open.

Over the last three decades, direct methods have been developed to reduce the computational expense, including  MUSIC \cite{Gruber2004}, linear sampling method \cite{Kirsch1996}, factorization method \cite{Kirsch1998} and direct sampling method (DSM) \cite{ItoJinZou:2012} etc.
These methods construct suitable indicator functionals for localizing inhomogeneities, and are mostly designed for elliptic inverse problems, typically relying on matrix-vector operations or range characterizations of certain operators.
The DSM employs a carefully chosen Sobolev dual product in the measurement space and a set of probing functions satisfying near orthogonality with respect to the dual product, and constructs index functions via dual products between measurements and probing functions. Since its first development \cite{ItoJinZou:2012}, the DSM has been extended to various elliptic equations \cite{ItoJinZou:2013,Zou2014,Zou2015,Zou2021}.
Very recently, the work \cite{Ito2025Iterative} proposed an iterative DSM (IDSM) for general elliptic inverse problems \cite{Ito2025Iterative}, delivering improved accuracy and stability with modest computational overhead.

On the contrary, the development of direct methods for time-dependent inverse problems remains relatively limited.
The linear sampling method has been adapted to recover the conductivity in the heat equation \cite{Sun2024,Heck2012,Nakamura2013} assuming sufficiently large data (i.e., the full DtN data) is available, and for time-domain inverse scattering, see the works \cite{CakoniRezac:2017} and \cite{GengSong:2025}, which require time-domain multistatic data (based on multiple incident directions).
The work \cite{ikehata2009} developed
the enclosure method, using carefully constructed boundary fluxes and the Dirichlet data, to identify internal cavities, but presented no numerical experiments.
The DSM has been extended to parabolic equations \cite{ChowItoZou:2018sisc} and time-domain wave problems \cite{guo2016time,guo2024novel}.
The works \cite{guo2016time,guo2024novel} treat wave equations characterized by finite propagation speeds and focus on recovering static inhomogeneities.
Thus, the setting differs markedly from parabolic inverse problems involving moving inhomogeneities, which is the main focus of this work.

So far direct methods still face significant challenges for parabolic inverse problems, which involve the following issues.
First, existing works on elliptic and parabolic problems \cite{Zou2014,Zou2015,Zou2021,ChowItoZou:2018sisc} construct the DSM functional using the adjoint operator of the direct problem and a multiplication operator.
This strategy is effective for elliptic problems, but its accuracy is inadequate in the transient scenarios where errors accumulate over time.
Second, the DSM focuses on reconstructing the source induced by the inhomogeneities, rather than directly targeting the inhomogeneities.
This neglects the regularity structure of the source, cf. Section \ref{sec:PRI}.
Third, existing works rely on fractional-order Sobolev inner products to enhance the orthogonality of probing functions, which however may lead to numerical instability in the presence of data noise.
Finally, existing approaches are restricted to linear potential inhomogeneities, and
nonlinear or mixed-type inhomogeneities have not been addressed.
To overcome these limitations, we extend the IDSM \cite{Ito2025Iterative} in the elliptic case and propose a novel IDSM for parabolic inverse problems. This approach relies on the following four key innovations:
\begin{itemize}
\item \textbf{Dynamic adaptation}:
The IDSM integrates an iterative low-rank correction mechanism that is dynamically adapted to the scattering field of the problem.
It significantly improves the accuracy of the index function with little increase in the computational overhead.

\item \textbf{Direct inhomogeneity imaging}:
Unlike conventional DSMs that indirectly image inclusions through the source induced by inclusions, the IDSM directly reconstructs inhomogeneities.
It enhances the reconstruction accuracy by addressing the issue of source regularity, cf. Section \ref{sec:PRI}.

\item \textbf{Robust formulation}:
We replace global fractional Sobolev inner products with localized $L^{2}$-inner products.
This avoids the numerical instability of evaluating Sobolev inner product and ensures stability under noise while preserving the resolution of the method.

\item \textbf{Abstract framework}:
We formulate the method in an abstract framework that is applicable to many scenarios, e.g., nonlinear problems (e.g., power-law potentials) and mixed-type inclusions (e.g., both conductivity and potential anomalies).
It allows treating diverse physical scenarios systematically.
\end{itemize}

While the abstract framework builds on the IDSM \cite{Ito2025Iterative} in the elliptic setting, its extension to the parabolic setting presents fundamental new challenges.
First, parabolic Green's functions exhibit distinct anisotropic decay properties (e.g., Gaussian-type spatial decay with temporal scaling) compared to elliptic kernels, necessitating a different probing strategy. 
Second, the error accumulates over an extended time horizon, which causes instability under noise and long-term evolution.
Using the dynamic adaptation of the low-rank structure of the kernel and robust formulation via the localized $L^2$ inner product, the proposed IDSM enables stably and accurately computing the index function for noisy data over an extended time horizon.
Third, parabolic inverse problems for recovering moving inhomogeneities often rely on one single boundary measurement.
This inherently restricts the resolution of the index function.
To address this issue, we incorporate a local iteration within each time segment, whose computational overhead is minimal, since the iteration takes only one step in most segments.
Numerical experiments confirm the method can reliably track complex moving inhomogeneities with minimal computational overhead (4-6 PDE solves) and handles nonlinear/mixed-type cases.
In sum, it can achieve high accuracy, robustness, and efficiency, and represents an effective tool for solving a class of parabolic inverse problems.

The rest of the paper is structured as follows.
In Section \ref{sec:MAT}, we formulate parabolic inverse problems in an abstract framework.
In Section \ref{sec:PRI}, we outline the general principle of the IDSM for parabolic inverse problems.
In Section \ref{sec:IMP}, we describe the development of the IDSM.
We present numerical experiments in Section \ref{sec:NUM}, and conclude with discussions in Section \ref{sec:CON}.
Throughout, we use calligraphic symbols to denote operators, and capitalized letters for the associated integral kernels.
For any Hilbert space $H$ and Banach space $X$, we denote the inner product on $H$ by $(\cdot, \cdot)_{H}$ and the duality pairing between the dual space $X'$ of $X$ and $X$ by $\langle \cdot, \cdot \rangle_{X', X}$.

\section{Mathematical model}\label{sec:MAT}
In this section, we describe the mathematical framework for parabolic inverse problems.
Let $\Omega\subset \mathbb{R}^{d}$ ($d=2,3$) be an open bounded domain with a smooth boundary $\Gamma$, and $n$ denotes the unit outward normal derivative to the boundary $\Gamma$.
Fix the final
time $T$. Consider the following parabolic equation:
\begin{equation}\label{eqn1}
\partial_{t} y - \Delta y + N(y)u = f, \quad  \text{in } \Omega \times (0, T),
\end{equation}
equipped with the following initial and boundary conditions
\begin{equation}\label{eqn:ibc}
\left\{
\begin{aligned}
{\partial_n} y &= g, & \text{on }& \Gamma \times (0, T), \\
y &= h, & \text{in }& \Omega \times \{0\},
\end{aligned}
\right.
\end{equation}
where $f$ and $g$ denote the external sources applied to the domain $\Omega$ and the boundary $\Gamma$, respectively, subject to suitable compatibility conditions that ensure the regularity of the solution $y$.
The time-independent operator $N$ characterizes the inhomogeneity, and $u$ describes its spatiotemporal distribution as a piecewise constant function in space:
\begin{equation}\label{eqn10}
u(x, t) = \sum_{j=1}^{J} p_{j}(t) \chi_{\omega_{j}(t)}(x).
\end{equation}
Here, $\{p_{j}(t)\}_{j=1}^{J}$ are distinct, unknown time-dependent functions, and $\{\chi_{\omega_{j}(t)}\}_{j=1}^{J}$ denote the characteristic functions of disjoint open sets $\{\omega_{j}(t)\}_{j=1}^{J}$.
The sets $\{\omega_{j}(t)\}_{j=1}^{J}$ represent the spatial supports of the inhomogeneities at time $t$, which evolve within $\Omega$ with smooth boundaries and do not intersect the boundary $\Gamma$, i.e., $\mathrm{supp}(u) \cap \left(\Gamma \times [0, T]\right) = \varnothing$.
The temporal evolution of $\{\omega_{j}(t)\}_{j=1}^{J}$ and $\{p_{j}(t)\}_{j=1}^{J}$ describes various dynamic behavior, e.g., motion, deformation, disappearance ($\omega_{j} \to \emptyset$), fading ($p_{j} \to 0$), topological separation (splitting into two components) and merging (of disconnected components).
The inverse problem is to recover the locations and shapes of the inclusions $\{\omega_{j}\}_{j=1}^{J}$ from one or several pairs of lateral Cauchy data $y_{d}$, where $y_{d} \in L^{2}(0,T; L^{2}(\Gamma))$ denotes the boundary measurement of the solution $y$ to problem \eqref{eqn1}--\eqref{eqn:ibc}.

The operator $N$ provides an abstract framework for modeling moving inhomogeneities.
For any $y \in H^{1}(\Omega)$, $N(y)$ acts linearly on $u(\cdot,t) \in L^{\infty}(\Omega)$ and maps to $(H^{1}(\Omega))'$.
For example, for $N(y)u = -\nabla \cdot (u \nabla y)$, \eqref{eqn1} reads
\begin{equation}\label{eqn11}
\partial_{t} y - \nabla \cdot \left((1 + u) \nabla y \right) = f,\quad \text{in } \Omega \times (0,T),
\end{equation}
Under the positivity condition $u \geq \varepsilon_{0} > -1$, \eqref{eqn11} describes an inverse conductivity problem.
Nonlinear inhomogeneities can also be included.
For example, the nonlinear potential problem
\begin{equation}\label{eqn13}
\partial_{t} y - \Delta y + u |y|^{p-2} y = f, \quad \text{in } \Omega \times (0,T),
\end{equation}
is subsumed by \eqref{eqn1} with the choice $N(y)u = u |y|^{p-2} y$.
This model arises in laser-induced thermotherapy \cite{shuailu2019uniqueness}, where the nonlinear term captures the thermal effects induced by focused laser energy.
Furthermore, the abstract framework covers the presence of multiple inclusions of distinct physical nature by allowing vector-valued $N(y)$ and $u$.
For example, let $N(y) = (N_{c}(y), {N_{p}}(y))$ and $u = (u_{c}, u_{p})^{\top}$, with $N_{c}(y)u_{c} = -\nabla \cdot (u_{c} \nabla y)$ and $N_{p}(y)u_{p} = u_{p} y$.
Thus, $N(y)u = N_{c}(y)u_{c} + N_{p}(y)u_{p} = -\nabla \cdot (u_{c} \nabla y) + u_{p} y$.
This yields the following composite model:
\begin{equation}\label{eqn6}
\partial_{t} y - \nabla \cdot \left((1 + u_{c}) \nabla y \right) + u_{p} y = f,\quad \text{in } \Omega \times (0,T),
\end{equation}
which involves both conductivity and potential inhomogeneities.
The abstract form of $N(y)$ not only ensures the applicability of the method to a broad range of problems but also offers a systematic framework to address complex scenarios, e.g., nonlinearities, and composite inhomogeneities.

\section{General principle of the IDSM for parabolic inverse problems}\label{sec:PRI}
In this section, we establish the general principle of the IDSM for parabolic inverse problems of form \eqref{eqn1}.
\subsection{DSM}
Let $y(u) \in L^{2}(0,T; H^{1}(\Omega))$ be the solution to problem \eqref{eqn1}--\eqref{eqn:ibc} with the inhomogeneity $u$, and $u_{*}$ the exact inhomogeneity profile.
Let $\mathcal{T}$ be the trace operator mapping from $H^{1}(\Omega)$ to $L^{2}(\Gamma)$ or from $L^2(0, T; H^{1}(\Omega))$ to $L^2(0, T;L^{2}(\Gamma))$ when time is involved.
The range $L^{2}(\Gamma)$ is chosen (instead of $H^{1/2}(\Gamma)$) to align with the regularity of the noisy data $y_d$ given by
\begin{equation*}
y_{d} = \mathcal{T}y(u_{*}) + \varepsilon,
\end{equation*}
where $\varepsilon \in L^{2}(0,T; L^{2}(\Gamma))$ denotes noise.
Let $y_{\emptyset} = \mathcal{T}y(0)$ be the boundary value of the background solution of \eqref{eqn1}--\eqref{eqn:ibc} (i.e., $u = 0$).
The scattering field, defined by
\begin{equation*}
y^{s} = y_{\emptyset} - \mathcal{T}y(u_{*}),
\end{equation*}
measures the change of the boundary data induced by $u_*$, and its noisy version $y^{s}_{d}$ is given by
\begin{equation*}
y^{s}_{d} = y_{\emptyset} - y_{d}.
\end{equation*}
We denote by $S = N(y(u_{*}))u_{*} \in L^{2}(0,T; H^{1}(\Omega)')$ the source induced by the inclusion $u_{*}$.
Then the difference $y(0) - y(u_{*})$ satisfies 
\begin{equation}\label{eqn14}
\partial_{t} \left(y(0) - y(u_{*})\right) - \Delta \left(y(0) - y(u_{*})\right) = S, \quad \text{in } \Omega \times (0,T), 
\end{equation}
equipped with zero initial and boundary conditions.
The dependence of $y(0) - y(u_{*})$ on $S$ is linear, no matter whether $N$ is linear or not.
The linear relation is represented by the operator $\mathcal{F}\colon L^{2}(0,T; H^{1}(\Omega)') \mapsto L^{2}(0,T; H^{1}(\Omega))$, i.e., $y(0) - y(u_{*}) = \mathcal{F}S$.
The corresponding kernel $F$ is closely related to the heat kernel $\Phi(x,t)$, given by
\begin{equation*}
\Phi(x,t) = \chi_{\{t > 0\}} \frac{1}{(4\pi t)^{d/2}} \exp\left(-\frac{\|x\|^2}{4t}\right),
\end{equation*}
where $\|\cdot\|$ denotes Euclidean vector norm.
$F$ satisfies 
\begin{equation}\label{eqn22}
\left\{
\begin{aligned}
\partial_{t} F(\cdot; x', t') - \Delta F(\cdot; x', t') &= 0, & \text{in }& \Omega \times (0, T), \\
{\partial_n} F(\cdot; x', t') &= -{\partial_n} \Phi(t - t', x - x'), & \text{on }& \Gamma \times (0, T), \\
F(\cdot; x', t') &= 0, & \text{in }& \Omega \times \{0\}.
\end{aligned}
\right.
\end{equation}
Note that $F$ is also Green's function of the parabolic problem on the cylinder $\Omega\times(0,T)$.
$F(x,t; x',t')$ encodes the response at $(x,t)$ induced by a Dirac source located at $(x',t')$.
The kernel $F(x,t;x',t')$ and the associated operator $\mathcal{F}$ are connected via
\begin{equation*}
w(x,t) = \mathcal{F}v(x,t)  = \int_{0}^{T} \int_{\Omega} F(x,t; x',t') v(x',t') \, \mathrm{d}x' \, \mathrm{d}t'.
\end{equation*}
The adjoint kernel $F^{*}$, corresponding to the adjoint operator $\mathcal{F}^{*}$, is given by
\begin{equation*}
v (x,t)= \mathcal{F}^{*}w(x,t) = \int_{0}^{T} \int_{\Omega} F(x',t'; x,t) w(x',t') \, \mathrm{d}x' \, \mathrm{d}t'.
\end{equation*}

Then the noisy scattering field $y^{s}_{d}$ can be expressed by
\begin{equation*}
y^{s}_{d}(x,t) = \int_{0}^{T} \int_{\Omega} F(x,t; x',t') s(x',t') \, \mathrm{d}x' \, \mathrm{d}t' + \varepsilon(x,t), \quad \text{for } (x,t) \in \Gamma\times(0,T),
\end{equation*}
or equivalently, in an operator form, $y^{s}_{d} = \mathcal{T}\mathcal{F}S + \varepsilon$.
Existing works \cite{Zou2015, Zou2014, Zou2022, Zou2021} on the DSM focus on inverting the linear operator $\mathcal{T}\mathcal{F}$ to recover $S$ from $y^{s}_{d}$.
The probing operator $\mathcal{F}^{\dag}\colon L^{2}(0,T; L^{2}(\Gamma)) \mapsto L^{2}(0,T; (H^{1}(\Omega))')$ plays a central role in the DSM, such that $\mathcal{F}^{\dag}\mathcal{T}\mathcal{F} \approx \mathcal{I}$ (the identity operator), and then the index function
\begin{equation*}
\eta = \mathcal{F}^{\dag}y^{s}_{d} = \mathcal{F}^{\dag}\mathcal{T}\mathcal{F}S + \mathcal{F}^{\dag}\varepsilon \approx S
\end{equation*}
localizes the support of $S$.
Since $S$ may share the support with $u$, it can identify spatial regions containing the inclusion $u_{*}$. The construction neglects the regularity structure of the source $S$:
for the conductivity inclusion, $S = -\nabla \cdot (u_{*}\nabla y(u_{*})) \in L^{2}(0,T; H^{1}(\Omega)')$ and $S \notin L^{2}(0,T; L^{2}(\Omega))$, but for the potential inclusion, $S = u_{*}y(u_{*}) \in L^{2}(0,T; L^{2}(\Omega))$.
The index function $\eta$ of $S$ ignores the regularity structure, which may degrade the recovery accuracy.

\subsection{The proposed IDSM}
The IDSM constructs an index function for the inhomogeneity $u$ rather than the induced source $S$.
It exploits a direct link between the scattering field $y^{s}$ and the inclusion $u$, which is represented by 
\begin{equation*}
y_{\emptyset} - y(u_{*}) = \mathcal{T}\mathcal{F}\mathcal{N}(y(u_{*}))u_{*},
\end{equation*}
where for any $y \in L^{2}(0,T; H^{1}(\Omega))$, $\mathcal{N}(y): L^{\infty}(0,T; L^{\infty}(\Omega))\to L^{2}(0,T; H^{1}(\Omega)')$ is defined by
\begin{equation*}
\left(\mathcal{N}(y)u\right)(x, t) = N(y(x, t))u(x, t).
\end{equation*}
This extends the static operator $N(y)\colon H^{1}(\Omega) \mapsto (H^{1}(\Omega))'$ in Section \ref{sec:MAT} to a spatiotemporal setting.
The composition $\mathcal{H}[u]:=\mathcal{T}\mathcal{F}\mathcal{N}(y(u))$ and its kernel $H_{u}(x,t;x',t')=F(x,t;x',t')\mathcal{N}(y(u))(x',t')$ are called the forward operator and forward kernel, respectively.
It replaces the linear operator $\mathcal{F}$ with a semi-linear one and directly connects the scattering field $y^{s}$ with the inclusion $u_{*}$:
\begin{equation*}
y^{s}=\mathcal{H}[u_{*}]u_{*}.
\end{equation*}
Let $\hat{u}$ be an estimate obtained through the IDSM.
Solving problem \eqref{eqn1}--\eqref{eqn:ibc} with $\hat{u}$ yields the predicted field $y(\hat{u})$, whose deviation from $y(0)$ must satisfy
\begin{equation*}
\partial_{t} \left(y(0)-y(\hat{u})\right) - \Delta \left(y(0)-y(\hat{u})\right)=N(y(\hat{u}))\hat{u},\quad \text{in }\Omega\times (0,T),
\end{equation*}
equipped with the zero initial and boundary conditions.
That is, the scattering field $y_{\emptyset} - \mathcal{T}y(\hat{u})$ corresponds precisely to $\mathcal{T}\mathcal{F}N(y(\hat{u}))\hat{u}$.
The recovered source $N(y(\hat{u}))\hat{u}$ matches exactly the structure of the source induced by the exact one $u_{*}$.
It encodes the source structure, and enables consistently utilizing structural properties of $S$.

To reconstruct the evolution of the inclusion $u_{*}$, we must invert the semi-linear operator $\mathcal{H}[u_{*}]$, which encodes the relationship between the scattering field $y^{s}$ and inclusions $u_{*}$:
\begin{equation*}
y^s(x,t) = \int_{0}^{T} \int_{\Omega} H_{u_{*}}(x,t; x',t') u_{*}(x',t') \, \mathrm{d}x' \, \mathrm{d}t', \quad \text{for } (x,t) \in \Gamma\times(0,T).
\end{equation*}
The IDSM employs a dual product $\langle \cdot, \cdot \rangle_{\Gamma \times [0,T]}$ over the data space and a probing kernel $K$
to construct the index function.
For each sampling point $(x,t)\in \Omega\times[0,T]$, the probing kernel defines a probing function $K_{x,t}\triangleq K(x,t;\cdot)$, which is designed to satisfy the following near-orthogonality in the dual product $\langle \cdot, \cdot \rangle_{\Gamma \times [0,T]}$:
\begin{equation} \label{eqn3}
\langle K(x,t; \cdot), H_{u_{*}}(\cdot; x', t') \rangle_{\Gamma \times [0,T]}
\begin{cases}
\approx 1, & \text{if } (x,t) \text{ is close to } (x',t'), \\
\ll 1, & \text{if } (x,t) \text{ is far from } (x',t').
\end{cases}
\end{equation}
This property motivates defining an index function $\eta$ via the probing operator $\mathcal{K}$:
\begin{align*}
\eta(x,t) =&\left(\mathcal{K}y^{s}_{d}\right)(x,t) = \langle K(x,t; \cdot), y^s_d \rangle_{\Gamma \times [0,T]} \\
=& \langle K(x,t; \cdot), y^s \rangle_{\Gamma \times [0,T]} + \langle K(x,t; \cdot), \varepsilon \rangle_{\Gamma \times [0,T]} \\
\approx& \sum_{i=1}^{N} \frac{1}{N} \langle K(x,t; \cdot), H_{u_{*}}(\cdot; x_i, t_i) \rangle_{\Gamma \times [0,T]} u_{*}(x_i, t_i).
\end{align*}
If the property \eqref{eqn3} holds, then $\eta$ takes value $O(1)$ near the inclusion support $\{(x_i, t_i)\} \subset \mathrm{supp}(u_{*})$ and decays rapidly elsewhere.
Moreover, the integral form of $\eta$ mitigates the impact of noise.

\subsubsection{The probing function $K$ and the kernel $R$}
Existing works \cite{ItoJinZou:2012,ChowItoZou:2018sisc,ItoJinZou:2013,Zou2014,Zou2015,Zou2021,Zou2022} on wave and elliptic systems  have shown that Green's functions exhibit a near-orthogonality property in certain Sobolev inner products.
Specifically, for distinct points $(x_1,t_1)$ and $(x_2,t_2)$, the following relation holds:
\begin{equation}\label{eqn9}
\langle F(\cdot;x_1,t_1), F(\cdot;x_2,t_2) \rangle_{\Gamma \times [0,T]} \ll 1, \quad \text{if } |(x_1, t_1) - (x_2, t_2)| \text{ is large,}
\end{equation}
where $\langle \cdot,\cdot \rangle_{\Gamma \times [0,T]}$ denotes a Sobolev inner product defined on the lateral parabolic boundary $\Gamma\times[0,T]$.
This property extends to the kernel $H_{u_{*}}$: if $ |(x_1, t_1) - (x_2, t_2)|$ is large, then
\begin{align*}
{}&\langle H_{u_{*}}(\cdot;x_1,t_1), H_{u_{*}}(\cdot;x_2,t_2) \rangle_{\Gamma \times [0,T]} \\
=& \langle \mathcal{N}(y(u_{*}))(x_{1},t_{1})F(\cdot;x_1,t_1), \mathcal{N}(y(u_{*}))(x_{2},t_{2})F(\cdot;x_2,t_2) \rangle_{\Gamma \times [0,T]}\\
=&\mathcal{N}(y(u_{*}))(x_{1},t_{1})\mathcal{N}(y(u_{*}))(x_{2},t_{2})\langle F(\cdot;x_1,t_1), F(\cdot;x_2,t_2) \rangle_{\Gamma \times [0,T]}\ll 1.
\end{align*}
We define the Gramian kernel $G_{u}$ associated with $H_{u}$ by
\begin{equation} \label{eqn16}
G_{u}(x,t; x',t') = \langle H_{u}(\cdot;x,t), H_{u}(\cdot;x',t') \rangle_{\Gamma \times [0,T]},
\end{equation}
which is symmetric and positive semidefinite.
Then we can define a resolver kernel $R$ of delta-type:
\begin{equation*}
R(x,t; x',t') = \frac{\delta_{(x-x', t-t')}}{G_{u_{*}}(x,t; x,t)}.
\end{equation*}
We construct the probing kernel $K$ by composing the adjoint of $H_{u_{*}}$ with $R$:
\begin{equation}\label{eqn8}
K(x,t; x',t') = \int_{\Omega} \int_{0}^{T} R(x,t; x'',t'') H_{u_{*}}(x',t'; x'',t'') \, \mathrm{d}x'' \, \mathrm{d}t''.
\end{equation}
Thus the relation \eqref{eqn3} is satisfied:
\begin{align*}
&\langle K(x,t; \cdot), H_{u_{*}}(\cdot; x', t') \rangle_{\Gamma \times [0,T]}\\
= &\frac{\langle H_{u_{*}}(\cdot;x,t), H_{u_{*}}(\cdot;x',t') \rangle_{\Gamma \times [0,T]}}{\langle H_{u_{*}}(\cdot;x,t), H_{u_{*}}(\cdot;x,t) \rangle_{\Gamma \times [0,T]}}
\begin{cases}
\approx 1, & \text{when } (x,t) \text{ is close to } (x',t'), \\
\ll 1, & \text{when } (x,t) \text{ is far from } (x',t').
\end{cases}
\end{align*}
Generally computing Green’s function at each sampling point is expensive.
Instead, the trace of Green’s function can be approximated by
$\frac{1}{G_{u_{*}}(x,t; x,t)} \approx d(x,\Gamma)^{\nu}$,
where the exponent $\nu>0$ is empirically determined.
The positivity of $G_{u_{*}}$ ensures that $R$ is positive and symmetric.
Further, the delta form of $R$ facilitates an efficient evaluation:
\begin{equation} \label{eqn7}
\int_{0}^{T} \int_{\Omega} R(x,t; x',t') v(x',t') \, \mathrm{d}x' \, \mathrm{d}t' = \frac{v(x,t)}{d(x,\Gamma)^{\nu}}.
\end{equation}

While this choice of $R$ is efficient, its accuracy is inadequate for time-dependent problems, due to error accumulation.
Thus, we modify the kernel $R$ to enhance the accuracy.
The structure of $K$ in \eqref{eqn8} and the relation \eqref{eqn3} necessitate that $R$ satisfies
\begin{equation} \label{eqn4}
\int_{0}^{T} \int_{\Omega} R(x,t; x'',t'') G_{u_{*}}(x'',t''; x',t') \, \mathrm{d}x'' \, \mathrm{d}t''
\begin{cases}
\approx 1, & \text{if } (x,t) \text{ is close to } (x',t'), \\
\ll 1, & \text{if } (x,t) \text{ is far from } (x',t').
\end{cases}
\end{equation}
That is, $R$ is an approximate inverse of $G_{u_{*}}$, i.e., $\mathcal{R}\mathcal{G}[u_{*}] \approx \mathcal{I}$, the identity operator.
Additionally, $R$ must retain the efficiency, since a full convolution over $\Omega \times [0,T]$ is expensive.
We propose a hybrid kernel structure that combines singular and low-rank components:
\begin{equation}\label{eqn:low-rank}
R(x,t;x',t') = D(x,t) \delta_{(x-x',t-t')} + \sum_{j=1}^{K} m_{j}(x,t) n_{j}(x',t').
\end{equation}
This formulation maintains the efficiency through separable operations:
\begin{equation*}
\int_{0}^{T}\!\!\! \int_{\Omega} R(x,t;x',t') v(x',t') \, \mathrm{d}x' \, \mathrm{d}t' = D(x,t) v(x,t) + \sum_{j=1}^{K} \left(\int_{0}^{T} \!\!\!\!\int_{\Omega} n_{j}(x',t') v(x',t') \, \mathrm{d}x' \, \mathrm{d}t' \right) m_{j}(x,t),
\end{equation*}
where $\{m_j, n_j\}_{j=1}^{K}$ is designed to ensure kernel symmetry and positive semi-definiteness.
Numerically, it greatly improves the accuracy compared to the kernel in \eqref{eqn7}, and enables stably evaluating the index functions $\eta$ over an extended time horizon.

\subsubsection{Addressing the nonlinearity of $\mathcal{H}$}
In practice, since the knowledge of $y(u_*)$ is limited to the boundary $\Gamma \times [0,T]$, directly computing $\mathcal{N}(y(u_*))$ in the forward kernel $H_{u_*}$ is infeasible.
To overcome this challenge, we define a fixed-point iterative scheme:
\begin{equation}\label{eqn23}
\eta^{k+1} = \mathcal{R}\mathcal{H}[u^k]^* y^{s}_{d}.
\end{equation}
Given an initial guess $u^0 = 0$, the $(k+1)$th index function of the IDSM is computed as
\begin{equation*}
\eta^{k+1} = \mathcal{R} \mathcal{N}(y(u^{k}))^{*}\left(\mathcal{T}\mathcal{F}\right)^{*}y^{s}_{d}.
\end{equation*}
The detail of the iteration is given in Section \ref{subsec:cro}.
After obtaining $\eta^{k+1}$, a pointwise projection operator $\mathcal{P}$ is applied, to enforce box constraints (i.e., prior knowledge of the inclusion), yielding $u^{k+1} = \mathcal{P}(\eta^{k+1})$.
For example, in the conductivity problem in Section \ref{sec:MAT}, it may take the form
\begin{equation*}
u_{c}^{k+1} = \mathcal{P}(\eta^{k+1}) = \max(\eta^{k+1}, -1 + \varepsilon_{0}),
\end{equation*}
where $\varepsilon_{0} > 0$ is small.
This ensures the positivity of the conductivity.
The iterative process terminates when the relative error between the boundary measurements of $y(u^k)$ and the data $y_d$ falls below a given tolerance $\epsilon_{\text{tol}}$.
The iterative process progressively refines the estimate $u^k$.

\section{Practical aspects of the IDSM}\label{sec:IMP}
Now we develop the key components of the IDSM.

\subsection{The selection of the dual product}\label{subsec:sdp}

One key component in the IDSM is a suitable dual product $\langle\cdot, \cdot\rangle_{\Gamma\times [0,T]}$, which ensures the near-orthogonality of probing functions, balances numerical stability under noise, and directly impacts the reconstruction accuracy. For elliptic inverse problems \cite{Zou2021,Zou2022,Zou2015,Zou2014,ChowItoZou:2018sisc}, fractional Sobolev inner products were employed to enhance the near-orthogonality, cf. \eqref{eqn9}.
However, it can lead to significant errors in the presence of data noise.
For $N(y)u = uy$, the source $S = uy$ lies in $L^{2}(0,T; L^{2}(\Omega))$.
From \eqref{eqn14}, the difference $y(0) - y(u_{*})$ belongs to $L^{2}(0,T; H^{1}(\Omega))$ \cite[Theorem 5, Chapter 7]{Evans2010}.
Thus, the scattering field $y^{s}$ lies in $L^{2}(0,T; H^{3/2}(\Gamma))$.
In practice, we only have access to the noisy scattering field $y^{s}_{d} = y^{s} + \varepsilon \in L^{2}(0,T; L^{2}(\Gamma))$.
This causes errors in the index function $\eta$ through the term $\langle K, \varepsilon \rangle_{\Gamma \times [0,T]}$:
If $\langle \cdot, \cdot \rangle_{\Gamma \times [0,T]}$ involves derivatives, the noise in $y^{s}_{d}$ can lead to numerical instability.
To mitigate this issue, we adopt the $L^{2}$-inner product $\langle \cdot, \cdot \rangle_{\Gamma \times [0,T]} = (\cdot, \cdot)_{L^{2}(0,T; L^{2}(\Gamma))}$.
To compensate the lack of the near-orthogonality, we divide the time interval $[0,T]$ into smaller segments, and observe the relation:
\begin{equation*}
\lim_{\delta t \to 0^{+}} \left(F(x_{1},t; \cdot), F(x_{2}, t; \cdot) \right)_{L^{2}(0,\delta t; L^{2}(\Gamma))} = 0, \quad \forall x_{1} \neq x_{2},
\end{equation*}
since $\lim_{t \to t'} F(x,t; x',t') = \delta(x - x')$.
Thus, for small $\delta t$, the kernel $F$ (and hence $H_{u}$) exhibits near-orthogonality in the $L^{2}$ sense over each short time segment.
To this end, we partition $[0,T]$ into subintervals $\{(n\delta t, (n+1)\delta t]\}_{n=0}^{T/\delta t - 1}$ and apply the IDSM sequentially to each segment. Numerical experiments show that a properly chosen $\delta t$ can achieve the near-orthogonality of the kernel $H_{u}$ and the robustness against noise. In practice, after computing $\eta$ on a given segment, we can solve \eqref{eqn1} with $u_{*} \approx\mathcal{P}(\eta)$.
The terminal value $y(\cdot, (n+1)\delta t)$ from the current segment is then used as the initial condition for the IDSM in $((n+1)\delta t, (n+2)\delta t]$.

The IDSM involves two "iterative" components: temporal segment iterations and the iterative scheme \eqref{eqn23} to address the nonlinearity and to update the low-rank terms in $\mathcal{R}$, cf. Section \ref{subsec:cro}.
Below we use the superscript $n,k$ for functions associated with the $k$th iteration within the segment $(n\delta t, (n+1)\delta t]$: $\eta^{n,k}$ denotes the index computed at the $k$th iteration in the $n$th time segment $(n\delta t, (n+1)\delta t]$ via the IDSM.
The solution to \eqref{eqn1} in the segment, for the inclusion parameter $u^{n,k}$ and initialized by the solution from the preceding segment (or by $h$ for $n=0$), is denoted by $y^{n}(u^{n,k})$.
The superscripts will be occasionally omitted below.

\subsection{The adjoint of $\mathcal{H}[u]$}\label{subsec:ahu}
We describe an efficient approach to evaluate the adjoint of $\mathcal{H}[u]$:
\begin{equation*}
\zeta^{n,k+1} = \mathcal{H}[u^{k}]^{*} y^{s}_{d}.
\end{equation*}
The function $\zeta^{n,k}$ denotes the $k$th local dual function in the $n$th time segment.
The derivation is presented for the first time segment $(0, \delta t)$, and it is similar for the remaining ones.
The operator $\mathcal{H}[u]$ is the composition of $\mathcal{N}(y(u))$, $\mathcal{F}$ and $\mathcal{T}$, and its adjoint is given by
\begin{equation*}
\mathcal{H}[u]^* = \mathcal{N}(y(u))^* \left(\mathcal{T}\mathcal{F}\right)^*.
\end{equation*}
The operator $\mathcal{F}$ is associated with Green's function $F$ defined in \eqref{eqn22}.
However, solving
\eqref{eqn22} for every sampling point $(x',t')$ is prohibitive.
We employ the following result, which facilitates efficiently computing the adjoint of $\mathcal{T}\mathcal{F}$ by solving only one backward problem.
\begin{lemma}\label{lemma1}
For $y \in L^{2}(0, \delta t; L^{2}(\Gamma))$, let $z\in L^{2}(0,\delta t;H^{1}(\Omega))$ solve the backward heat equation:
\begin{equation}\label{eqn5}
\left\{
\begin{aligned}
\partial_{t} z + \Delta z &= 0, & \text{in }& \Omega \times (0, \delta t), \\
{\partial_n} z &= y, & \text{on }& \Gamma \times (0, \delta t), \\
z &= 0, & \text{in }& \Omega \times \{\delta t\}.
\end{aligned}
\right.
\end{equation}
Then, $z(x, t) = ( F(\cdot; x, t), y )_{L^{2}(0, \delta t; L^{2}(\Gamma))} = (\mathcal{T}\mathcal{F})^{*} y$.
\end{lemma}
\begin{proof}
It suffices to prove the relation $( z, v )_{L^{2}(0, \delta t; L^{2}(\Omega))} = ( y, \mathcal{T}\mathcal{F} v )_{L^{2}(0, \delta t; L^{2}(\Gamma))}$ for all smooth functions $v$ defined on $\Omega\times(0,T)$.
Let $w$ solve
\begin{equation*}
\left\{
\begin{aligned}
\partial_{t} w - \Delta w &= v, & \text{in }& \Omega \times (0, \delta t), \\
{\partial_n} w &= 0, & \text{on }& \Gamma \times (0, \delta t), \\
w &= 0, & \text{in }& \Omega \times \{0\}.
\end{aligned}
\right.
\end{equation*}
By the definition of $\mathcal{F}$, we have $w = \mathcal{F} v$.
By integration by parts, we derive
\begin{align*}
( z, v )_{L^{2}(0, \delta t; L^{2}(\Omega))} &= \int_{0}^{\delta t} \int_{\Omega} z v \, \mathrm{d}x \, \mathrm{d}t = \int_{0}^{\delta t} \int_{\Omega} z \left(\partial_{t} w - \Delta w\right) \, \mathrm{d}x \, \mathrm{d}t \\
&= -\int_{0}^{\delta t} \int_{\Omega} w \partial_{t} z \, \mathrm{d}x \, \mathrm{d}t - \int_{0}^{\delta t} \left(\int_{\Omega} w \Delta z \, \mathrm{d}x + \int_{ \Gamma} {\partial_n} z \, w \, \mathrm{d}x\right) \, \mathrm{d}t \\
&= \int_{0}^{\delta t} \int_{ \Gamma} y w \, \mathrm{d}x \, \mathrm{d}t = \int_{0}^{\delta t} \int_{ \Gamma} y \mathcal{F} v \, \mathrm{d}x \, \mathrm{d}t
= ( y, \mathcal{T}\mathcal{F} v )_{L^{2}(0, \delta t; L^{2}(\Gamma))},
\end{align*}
which proves the relation $z = (\mathcal{T}\mathcal{F})^{*} y$.
\end{proof}

Next we derive the adjoint of $\mathcal{N}(y)$ for the case with $L$ inhomogeneity types $\{N_{\ell}\}_{\ell=1}^{L}$, and one measurement associated with $(f, g, h)$.
Then
$N(y) = \left(N_{1}(y), \dots, N_{L}(y)\right)$ and $ u = \left(u_{1}, \dots, u_{L}\right)^{\top}$,
where each  $u_{\ell}$ is of the form \eqref{eqn10}.
The governing equation \eqref{eqn1} reads
\begin{equation*}
\partial_{t} y - \Delta y + \sum_{\ell=1}^{L} N_{\ell}(y) u_{\ell} = f, \quad \text{in } \Omega \times (0, \delta t)
\end{equation*}
The forward operator $\mathcal{H}[u]$ is given by
\begin{equation*}
y^{s} = y_{\emptyset} - \mathcal{T} y_{*} = \mathcal{T}\mathcal{F} \sum_{\ell=1}^{L} N_{\ell}(y) u_{\ell}.
\end{equation*}
To compute $\zeta = \mathcal{H}[u]^{*} y^{s}$, we first solve \eqref{eqn5} for the scattering field $y^{s}$.
This step evaluates the action of the adjoint of $\mathcal{T}\mathcal{F}$, yielding $z = \left(\mathcal{T}\mathcal{F}\right)^{*} y^{s}$.
The adjoint of $\mathcal{H}[u]$ is then given by
\begin{equation*}
\begin{pmatrix}
\zeta_{1} \\
\vdots \\
\zeta_{L}
\end{pmatrix}
\triangleq
\mathcal{H}[u]^{*}y^s
= \begin{pmatrix}
\mathcal{N}_{1}(y(u))^{*} \\
\vdots  \\
\mathcal{N}_{L}(y(u))^{*} 
\end{pmatrix}
z .
\end{equation*}
Thus, for the $\ell$th type  inclusion, the $\ell$th local dual function is $\zeta_{\ell} =\mathcal{N}_{\ell}(y(u))^{*} z$.

For the adjoint operator $\mathcal{N}_{\ell}(y(u))^{*}$, we further illustrate the derivation on the model \eqref{eqn6} in Section \ref{sec:MAT}, for which $N_{c}(y)u = -\nabla \cdot (u \nabla y)$ for the conductivity inclusion and $N_{p}(y)u = u y$ for the potential inclusion.
For any $p \in L^{2}(0,T; H^{1}(\Omega))$, the duality pairings associated with $\mathcal{N}_{c}$ and $\mathcal{N}_{p}$ are given by
\begin{equation*}
\langle \mathcal{N}_{c}(y)u, p \rangle = \int_{0}^{T} \int_{\Omega} u \nabla p \cdot \nabla y \, \mathrm{d}x \, \mathrm{d}t \quad \text{and} \quad \langle \mathcal{N}_{p}(y)u, p \rangle = \int_{0}^{T} \int_{\Omega} u p y \, \mathrm{d}x \, \mathrm{d}t,
\end{equation*}
respectively,
where $\langle \cdot, \cdot \rangle$ denotes the duality pairing between $L^{2}(0,T; H^{1}(\Omega))$ and its dual space.
These identities yield the adjoint actions of $\mathcal{N}_{c}(y)^{*}$ and $\mathcal{N}_{p}(y)^{*}$:
\begin{equation*}
\begin{aligned}
\langle \mathcal{N}_{c}(y)^{*} p, u \rangle &= \int_{0}^{T} \int_{\Omega} u \nabla p \cdot \nabla y \, \mathrm{d}x \, \mathrm{d}t \quad \Rightarrow \quad \mathcal{N}_{c}(y)^{*} p = \nabla p \cdot \nabla y, \\
\langle \mathcal{N}_{p}(y)^{*} p, u \rangle &= \int_{0}^{T} \int_{\Omega} u p y \, \mathrm{d}x \, \mathrm{d}t \quad \Rightarrow \quad \mathcal{N}_{p}(y)^{*} p = p y,
\end{aligned}
\end{equation*}
where $\langle \cdot, \cdot \rangle$ denotes the duality pairing between $L^{\infty}((0, T) \times \Omega)$ and its dual space.
After solving the backward problem \eqref{eqn5} for the scattering field $y^{s}$, the local dual function $\zeta$ is given by
\begin{equation}
\zeta = \mathcal{H}[u]^{*}
y^{s}
= \begin{pmatrix}
\zeta_{c} \\
\zeta_p
\end{pmatrix},
\quad \text{with }
\zeta_{c} = \nabla z \cdot \nabla y(u),\,\,
\zeta_p = z y(u).
\label{eqn15}
\end{equation}
Thus, $\mathcal{H}[u]^{*}$ generates distinct index functions for each type of inhomogeneity, and the conductivity and potential inclusions are reconstructed separately via $\zeta_c = \nabla z \cdot \nabla y(u)$ and $\zeta_p = z y(u)$, which mitigates cross-coupling artifacts.
The decoupling is essential for accurately resolving composite models involving multiple physical anomalies; see Section \ref{sec:NUM} for numerical illustrations.

\subsection{The construction of $\mathcal{R}$}\label{subsec:cro}
Now we construct the low-rank correction to the kernel $R$.
Given the evolution nature of parabolic equations, the IDSM constructs $\{m_j, n_j\}_{j=1}^K$ in \eqref{eqn:low-rank} on the fly.
To motivate the choice, we define an auxiliary scattering field $\hat{y}^s$ associated with $\hat{u}$ by $\hat{y}^s = y_\emptyset - \mathcal{T}y(\hat{u})$.
Following the preceding derivation, we obtain
\begin{equation*}
\hat{y}^s = \mathcal{T}\mathcal{F}\mathcal{N}(y(\hat{u}))\hat{u} = \mathcal{H}[\hat{u}]\hat{u}.
\end{equation*}
We compute the auxiliary local dual function $\hat{\zeta}$ for $\hat{y}^{s}$ as
$\hat{\zeta} = \mathcal{H}[\hat{u}]^* \hat{y}^s$. This establishes the identity $\hat{\zeta} = \mathcal{H}[\hat{u}]^* \hat{y}^{s} = \mathcal{G}[\hat{u}]\hat{u}$, where $\mathcal{G}[u]$ denotes the Gramian operator, cf. \eqref{eqn16}.
The auxiliary pair $(\hat{u}, \hat{\zeta})$ allows refining the kernel $\mathcal{R}$ using the relation
\begin{equation}\label{eqn17}
\mathcal{P}(\mathcal{R}\hat{\zeta}) = \hat{u}.
\end{equation}
This relation \eqref{eqn4} requires $\mathcal{R}$ to approximate the inverse of $\mathcal{G}[u_*]$.
In \eqref{eqn17}, $\mathcal{P}\circ\mathcal{R}$ invert $\mathcal{G}[\hat{u}]$ exactly: it provides an exact prediction $\mathcal{P}(\mathcal{R}\hat{\zeta})$ for the problem involving $\hat{u}$, which may be used to construct a more accurate index function $\eta = \mathcal{R}\zeta$.
However, the relation \eqref{eqn17} alone is insufficient to uniquely determine $\mathcal{R}$.
We employ dynamic low-rank corrections to enrich the kernel ${R}$.

In the IDSM, we construct the auxiliary problem on the fly in order to construct the low-rank correction.
Within the iterative scheme \eqref{eqn23}, the scattering field of the inhomogeneity $u^{n,k}$ may be used to update the kernel $R$ during the $k$th iteration of the $n$th time segment.
Specifically, in the $n$th time segment $(n\delta t, (n+1)\delta t]$, we first initialize $u^{n,0}=0$ and compute the background solution $y^{n}(0)$, using the initial value provided by the previous segment (or $h$ for the first segment).
This enables computing $y^{s,n}_{d} = y_{\emptyset}^{n} - y_{d}^{n}$.
Following the discussion in Section \ref{subsec:ahu}, the first local dual function $\zeta^{n,1}$ is derived in two steps.
First, we solve \eqref{eqn5} to obtain $z^{n}$, and then compute
\begin{equation}\label{eqn18}
\zeta^{n,1} = \mathcal{N}(y^{n}(u^{n,0}))^{*} z^{n}= \mathcal{N}(y^{n}(0))^{*} z^{n},
\end{equation}
where $y^{n}(u^{n,0})$ approximates the ground truth $y^{n}(u_{*}^{n})$ in $\mathcal{N}(y^{n}(u^{n,0}))^{*}$.
The first index function $\eta^{n,1}$ in the time segment is then computed as
\begin{align}
\eta^{n,1}(x, t) &= \int_{n\delta t}^{(n+1)\delta t} \int_{\Omega} R(x, t; x', t') \zeta^{n,1}(x, t) \, \mathrm{d}x' \, \mathrm{d}t' \notag \\
&= D(x, t) \zeta^{n,1}(x, t) + \sum_{j=1}^{K} \int_{n\delta t}^{(n+1)\delta t} \int_{\Omega} n_{j}(x', t') \zeta^{n,1}(x', t') \, \mathrm{d}x' \, \mathrm{d}t' \, m_{j}(x, t). \label{eqn21}
\end{align}
In practice, $R$ is initialized to $R = D(x, t) \delta_{x - x', t - t'}$.
Using the index function $\eta^{n,1}$, we obtain an estimate $u^{n,1}=\mathcal{P}(\eta^{n,1})$ of $u_{*}^{n}$, and compute \eqref{eqn1} over the interval $(n\delta t, (n+1)\delta t]$ to compute $y^{n}(u^{n,1})$.
If the relative error $\|y^{n}(u^{n,1}) - y_{d}^{n}\|_{L^{2}}/\|y_{d}^{n}\|_{L^{2}}$ falls below a predefined tolerance $\epsilon_{\text{tol}}$, $u^{n,1}$ is deemed an accurate estimate on the current segment.
However, if it exceeds $\epsilon_{\text{tol}}$, $\mathcal{R}$ is insufficiently accurate as an inverse of $\mathcal{G}[u_{*}^{n}]$.
Then an auxiliary scattering field $\hat{y}^{s,n,1}$ is defined by $$\hat{y}^{s,n,1} := y_{\emptyset}^{n} - \mathcal{T} y^{n}(u^{n,1}) = \mathcal{H}[u^{n,1}] u^{n,1}.$$
We obtain $\hat{z}^{n,1}$ by solving problem \eqref{eqn5} again with the boundary flux $\hat{y}^{s,n,1}$, and then compute the first auxiliary local dual function $\hat{\zeta}^{n,1} = \mathcal{H}[u^{n,1}]^{*} \hat{y}^{s,n,1} = \mathcal{G}[u^{n,1}] u^{n,1}$.

To maintain the relation \eqref{eqn4}, we update $R$ using the relation $\mathcal{P}(\mathcal{R} \hat{\zeta}^{n,1}) = u^{n,1}$.
In the work \cite{Ito2025Iterative}, motivated by the analogy between the relation $\mathcal{P}(\mathcal{R} \hat{\zeta}^{n,1}) = u^{n,1}$ and the secant equation in quasi-Newton methods, we employ Broyden-type low-rank update, e.g., Davidon–Fletcher–Powell (DFP) and Broyden–Fletcher–Goldfarb–Shanno (BFGS) methods, to enrich the low-rank structure of the kernel ${R}$. In quasi-Newton methods, the Hessian approximation $H^{k+1}$ satisfies the secant condition $H^{k+1}\delta g^{k} = \delta x^{k}$, where $\delta g^{k} = g^{k+1} - g^{k}$ and $\delta x^{k} = x^{k+1} - x^{k}$ denote incremental changes in gradients and independent variables, respectively.
The Broyden type updates take the form (cf. \cite[p. 149]{Wright2006} and \cite[p. 169]{Schnabel1996}):
\begin{align*}
(\text{DFP})\quad H^{k+1} &= H^{k} + \frac{\delta x^{k} \delta x^{k\top}}{\delta g^{k\top} \delta x^{k}} - \frac{H^{k} \delta g^{k} (H^{k}\delta g^{k})^\top }{\delta g^{k\top} H^{k} \delta g^{k}} ,\\
(\text{BFG})\quad H^{k+1} &= H^k + \left(1+\frac{\delta g^{k\top}H^{k}\delta g^{k}}{\delta g^{k\top} \delta x^k}\right) \frac{\delta x^k \delta x^{k\top}}{\delta g^{k\top} \delta x^k} - \frac{\delta x^k (H^k\delta g^k)^\top + H^k \delta g^k \delta x^{k\top}}{\delta g^{k\top} \delta x^k}.
\end{align*}
Analogously, in the IDSM, the low-rank components of $\mathcal{R}$ are defined by
\begin{align}
\delta R_{\rm DFP}(x,t;x',t') &= \frac{\hat{\eta}^{n,1}(x,t)\hat{\eta}^{n,1}(x',t')}{(\hat{\zeta}^{n,1},\hat{\eta}^{n,1})}-\frac{\mathcal{R}\hat{\zeta}^{n,1}(x,t)\mathcal{R}\hat{\zeta}^{n,1}(x',t')}{(\hat{\zeta}^{n,1},\mathcal{R}\hat{\zeta}^{n,1})},\label{eqn:DFP}\\
\delta R_{\rm BFG}(x,t;x',t') &= \left(1+\frac{(\hat{\zeta}^{n,1},\mathcal{R}\hat{\zeta}^{n,1})}{(\hat{\zeta}^{n,1}, \hat{\eta}^{n,1})}\right)\frac{\hat{\eta}^{n,1}(x,t)\hat{\eta}^{n,1}(x',t')}{(\hat{\zeta}^{n,1}, \hat{\eta}^{n,1})} \nonumber \\
&\quad -\frac{\hat{\eta}^{n,1}(x,t)\mathcal{R}\hat{\zeta}^{n,1}(x',t') + \mathcal{R}\hat{\zeta}^{n,1}(x,t)\hat{\eta}^{n,1}(x',t')}{(\hat{\zeta}^{n,1}, \hat{\eta}^{n,1})},\label{eqn:BFG}
\end{align}
where $(\cdot, \cdot)$ denotes the $L^{2}(n\delta t, (n+1)\delta t; L^{2}(\Omega))$ inner product, and $\hat{\eta}^{n,1}$ is defined by
\begin{equation}\label{eqn20}
\hat{\eta}^{n,1}\in\arg\min\|\hat{\eta}^{n,1}-\mathcal{R}\hat{\zeta}^{n,1}\|_{L^2(\Omega)}^{2} \quad \text{subject to} \quad \mathcal{P}(\hat{\eta}^{n,1})=u^{n,1}.
\end{equation}
These corrections are computationally efficient \cite[Chapter 6]{Wright2006} and preserve the symmetry and positive definiteness of the kernel $R$. Broyden's methods produce identical iterative sequences for the objective function, irrespective of whether the function is quadratic or not \cite[Theorem 8.2.2]{Schnabel1996}.
Thus the choice of the correction is not important; see the numerical experiments in Section \ref{sec:NUM}.

Following the update of $R$, a refined index function $\eta^{n,2}$ is computed.
$z^{n}$ (cf. \eqref{eqn5}), dependent solely on the scattering field $y^{s}_{d}$, does not require recomputation.
Thus the process begins with evaluating the second local dual function $\zeta^{n,2} = \mathcal{N}(y^{n}(u^{n,1}))^{*} z^{n}$, where $\mathcal{N}(y^{n}(u^{n,1}))^{*}$ replaces $\mathcal{N}(y^{n}(u^{n,0}))^{*}$ in \eqref{eqn18} to provide a more precise approximation.
Using the updated $\mathcal{R}$, we compute $\eta^{n,2} = \mathcal{R} \zeta^{n,2}$ and get an improved estimate $u^{n,2}=\mathcal{P}(\eta^{n,2})$ of  $u_{*}^{n}$.
The relative error of $y^{n}(u^{n,2})$ is then evaluated to determine whether it falls below the tolerance $\epsilon_{\text{tol}}$.
If the error exceeds $\epsilon_{\text{tol}}$, the process is repeated with $k=3$ and then augment the kernel $R$ with an additional low-rank term.
The refinement continues until a satisfactory estimate of $u_{*}^{n}$ is achieved.

\subsection{Robust implementation}\label{subsec:srs}
When forming the index function $\eta = \mathcal{R}\zeta$, practical challenges still arise due to the presence of noise, particularly over an extended time horizon.
We integrate three efficient strategies into the algorithm to enhance the numerical stability.

\begin{itemize}
\item \textbf{Mesh coarsening}:
We employ mesh coarsening to reduce memory requirements and mitigate the impact of noise.
Since $u$ and $\eta$ approximate piecewise-constant functions, and the operator $\mathcal{N}$ maps $L^\infty(0,T;L^\infty(\Omega))$ to $L^2(0,T;H^1(\Omega)')$, it follows that $u(\cdot,t)$, $\eta(\cdot,t)$ and $\zeta(\cdot,t)$ all constitute a finitely additive, absolutely continuous measure with respect to Lebesgue measure.
Thus, we compute $u$, $\eta$, $\zeta$ and the basis functions $\{m_j, n_j\}_{j=1}^K$ in $R$ using a coarse mesh.
In Section \ref{sec:NUM}, the finite element method is implemented with two separate meshes: a fine mesh $\mathscr{T}_{f}$ and a coarse mesh $\mathscr{T}_{c}$.
For each time segment, the backward heat equation \eqref{eqn5} is solved on the fine mesh $\mathscr{T}_{f}$ using piecewise linear elements.
Then, the local dual function $\zeta$ is computed on the coarse mesh $\mathscr{T}_{c}$ using piecewise constant elements via averaging.
This strategy effectively suppresses noise through averaging.

\item \textbf{Damping low-rank components}:
We incorporate a damping factor $\lambda \in (0,1)$ to the low-rank terms since the Gramian $\mathcal{G}[u] = \mathcal{H}[u]^{*}\mathcal{H}[u]$ depends on $u$.
The auxiliary pair $(\hat{u}^{n}, \hat{\zeta}^{n})$ may lose accuracy in later time segments $(N \delta t, (N + 1)\delta t]$ when $N \gg n$.
Thus, a damping factor $\lambda\in(0,1)$ is applied to the low-rank term at the end of each time segment: $m_{j} \leftarrow \lambda m_{j}$, so that the influence of outdated information diminishes over time.
It becomes negligible after several time segments, and we limit its total rank.
Moreover, at the first iteration within each time segment, the singular component $D(x,t)$ is rescaled by the factor ${\|u^{n,1}\|_{L^{1}(\Omega)}}/{\|D\hat{\zeta}^{n,1}\|_{L^{1}(\Omega)}}$ in order to account for temporal variation of  $\mathcal{G}[u]$.

\item \textbf{Dirichlet boundary condition for initial values}:
The third strategy computes the initial value for the next time segment using a Dirichlet boundary condition in place of the Neumann one by solving the following modified equation (instead of \eqref{eqn1}--\eqref{eqn:ibc}):
\begin{equation}\label{eqn2}
\left\{
\begin{aligned}
\partial_{t} y^{n} - \Delta y^{n} + N(y^{n})u^{n} &= f, & \text{in }& \Omega \times (n\delta t, (n+1)\delta t), \\
y^{n} &= y_{d}^{n}, & \text{on }& \Gamma \times (n\delta t, (n+1)\delta t), \\
y^{n} &= y^{n-1}, & \text{in }& \Omega \times \{n\delta t\}.
\end{aligned}
\right.
\end{equation}
The solution $y_{D}^{n}(u^{n})$  to \eqref{eqn2} has better accuracy than the solution to \eqref{eqn1}, $y_{D}^{n}(u^{n})$, particularly near the boundary $\Gamma$.
This gives a more accurate initial condition for the next time segment, thereby enhancing the long-term stability of the method.
\end{itemize}

\subsection{The algorithm for the IDSM}\label{subsec:aid}

Now we provide a detailed algorithm for the IDSM; see the pseudocode in Algorithm \ref{alg1}. We illustrate the IDSM on the model in \eqref{eqn6}.

\begin{algorithm}[hbt!]
\caption{The IDSM for the parabolic inverse problem}
\label{alg1}
\begin{algorithmic}[1]
\Require Boundary measurements $y_{d}$, time interval $[0, T]$, time step $\delta t$, tolerance $\epsilon_{\text{tol}}$, damping factor $\lambda$, and initial kernel $R(x,t;x',t')=D(x,t)\delta_{x,t}(x',t')$.

\State \textbf{Initialization}: Set the initial kernel $\mathcal{R}$.
\For{$n = 0, 1, 2, \ldots, \frac{T}{\delta t}-1$} \Comment{Iterate over time segments}
\State Solve the background equation \eqref{eqn1} with initial data $y_{D}^{n-1}(u^{n-1})$ to obtain $y_{\emptyset}^{n}=\mathcal{T}y^{n}(0)$.
\State Initialize the current segment solution: $y^{n} \leftarrow y^{n}(0)$.

\State Compute the scattering field: $y^{s,n}_{d} = y_{\emptyset}^{n} - y_{d}^{n}$.
\State Solve the adjoint problem \eqref{eqn5} with Neumann data $y^{s,n}_{d}$ to compute $z^{n}$.

\For{$k = 1, 2, 3, \ldots$} \Comment{Iterate until convergence}
\State Compute the $k$th local dual function: $\zeta^{n,k} = \mathcal{N}(y^{n})^{*} z^{n}$.
\State Project $\zeta^{n,k}$ onto the coarse mesh. 

\State Form the $k$th estimate: $u^{n,k} = \mathcal{P}(\mathcal{R}\zeta^{n,k})$..

\State Obtain $y^{n}=y^{n}(u^{n,k})$ by solving \eqref{eqn1} with $u^{n,k}$.

\State Compute the $k$th auxiliary scattering field: $\hat{y}^{s,n,k} = y_{\emptyset}^{n} - \mathcal{T}y^{n}$.
\State Compute $\hat{z}^{n,k}$ by solving problem \eqref{eqn5} with the boundary flux $\hat{y}^{s,n,k}$.
\State Compute the $k$th auxiliary local dual function $\hat{\zeta}^{n,k} = \mathcal{H}[u^{n,k}]^{*} \hat{y}^{s,n,k}$.
\State Compute the low-rank update to $\mathcal{R}$ using either DFP \eqref{eqn:DFP} or BFG \eqref{eqn:BFG}.
\If{$k=1$}
\State $D(x,t)=D(x,t){\|u^{n,1}\|_{L^{1}(\Omega)}}/{\|D\hat{\zeta}^{n,1}\|_{L^{1}(\Omega)}}$.
\EndIf
\If{$\|y^{n} - y_{d}^{n}\|_{L^{2}}/\|y_{d}^{n}\|_{L^{2}} \leq \epsilon_{\text{tol}}$}
\State \textbf{Break} the inner loop.
\EndIf
\EndFor

\State Compute the final local dual function $\zeta^{n} = \mathcal{N}(y^{n})^{*} z^{n}$ and project onto the coarse mesh $mathscr{T}_{c}$.
\State Obtain the estimated inhomogeneity: $u^{n} = \mathcal{P}(\mathcal{R}\zeta^{n})$.
\State Solve the direct problem \eqref{eqn2} with $u^{n}$ to get $y_{D}^{n}(u^{n})$.

\State Damp low-rank terms: $m_{j} \leftarrow \lambda m_{j}$ for all $j>0$.
\EndFor

\State \Return The reconstructed inhomogeneity $\{u^{n}\}_{0}^{T/\delta t-1}$ and solutions $\{y_{D}^{n}(u^{n})\}_{0}^{T/\delta t-1}$.
\end{algorithmic}
\end{algorithm}

\begin{example}
This example illustrates the IDSM on recovering both conductivity and potential inclusions, cf. \eqref{eqn6}, from one pair of Cauchy data $y_{d}$ generated by the triple $(f, g, h)$. We illustrate the algorithm in the first time segment step by step.

\begin{itemize}
\item \textbf{Step 3-6: Compute the background solution and scattering field.}

First, we obtain $y_{\emptyset}^0 = \mathcal{T} y^0(0)$ by solving the background model:
\begin{equation*}
\partial_{t} y^{0}(0) - \Delta y^{0}(0) = f, \quad \text{in }\Omega \times (0, \delta t),
\end{equation*}
equipped with the initial condition $h$ and boundary condition $g$, cf. \eqref{eqn:ibc}.
This yields the scattering field $y_{d}^{s,0} = y_{\emptyset}^0 - y_{d}^0$.
Next, we solve problem \eqref{eqn5} with the boundary flux $y_{d}^{s,0}$:
\begin{equation*}
\left\{
\begin{aligned}
\partial_{t} z^{0} + \Delta z^{0} &= 0, & \text{in }& \Omega \times (0, \delta t), \\
{\partial_n} z^{0} &= y_{d}^{s,0}, & \text{on }& \Gamma \times (0, \delta t), \\
z^{0} &= 0, & \text{in }& \Omega \times \{\delta t\}.
\end{aligned}\right.
\end{equation*}
By Lemma \ref{lemma1}, $z^0 = (\mathcal{T}\mathcal{F})^* y_{d}^{s,0}$.
$z^0$ is independent of the inclusion $u$.

\item \textbf{Step 7-22: Iterative refinement}

We compute the first local dual function $\zeta^{0,1} = (\zeta_c^{0,1}, \zeta_p^{0,1})^\top$  using \eqref{eqn15}:
\begin{equation*}
\zeta_c^{0,1} = \nabla z^0 \cdot \nabla y^0(0)\quad\text{and}\quad
\zeta_p^{0,1} =  z^0 y^0(0).
\end{equation*}
Then we project the local dual function $\zeta^{0,1}$  onto a coarse mesh via averaging, and compute
the index function $\eta^{0,1} = (\eta_c^{0,1}, \eta_p^{0,1})^\top$ by
\begin{equation}\label{eqn19}
\eta_c^{0,1}(x,t) = D_c(x,t) \zeta_c^{0,1}(x,t)\quad\text{and}\quad
\eta_p^{0,1}(x,t) = D_p(x,t) \zeta_p^{0,1}(x,t),
\end{equation}
where $D_c$ and $D_p$ denote the initial kernel $R$.
We obtain the first estimate $u^{0,1} = \mathcal{P}(\eta^{0,1})$ of $u_*^0$ by pointwise projection:
\begin{equation*}
\hat{u}_c^{0,1} = \max(\eta_c^{0,1}, -1 + \varepsilon_0)\quad\text{and}\quad
\hat{u}_p^{0,1} = \max(\eta_v^{0,1}, 0.0).
\end{equation*}
Then, we compute $y^0(u^{0,1})$ and the relative error $e = {\|y^0(u^{0,1}) - y_{d}^0\|_{L^2}}/{\|y_{d}^0\|_{L^2}}$.
If $ e > \epsilon_{\text{tol}}$, the estimate $u^{0,1}$ is insufficiently accurate, and an auxiliary pair $(u^{0,1}, \hat{\zeta}^{0,1})$ is constructed.
To get the auxiliary dual $\hat{\zeta}^{0,1}$, we solve the backward heat equation with $\hat{y}^{s,0} = y_{\emptyset}^0 - \mathcal{T} y^0(u^{0,1})$ and obtain $\hat{z}^{0,1}$.
Then we compute $\hat{\zeta}^{0,1} = (\hat{\zeta}_c^{0,1}, \hat{\zeta}_p^{0,1})^\top$ by
\begin{equation*}
\hat{\zeta}_c^{0,1} = \nabla \hat{z}^{0,1} \cdot \nabla y^0(u^{0,1})\quad\text{and}\quad
\hat{\zeta}_p^{0,1} = \hat{z}^{0,1} y^0(u^{0,1}).
\end{equation*}
We add a new low-rank term to the  kernel $R$ using either \eqref{eqn:DFP} or \eqref{eqn:BFG}, where $\hat{\eta}^{0,1}$ is computed via its definition in \eqref{eqn20}:
\begin{align*}
\hat{\eta}_{c}^{0,1}(x,t) &= \begin{cases}
\hat{u}_{c}^{0,1}(x,t), & \text{if } \hat{u}_{c}^{0,1}(x,t) > -1 + \varepsilon_{0}, \\
\min\{\hat{\zeta}_{c}^{0,1}(x,t), -1 + \varepsilon_{0}\}, & \text{if } \hat{u}_{c}^{0,1}(x,t) = -1 + \varepsilon_{0},
\end{cases} \\
\hat{\eta}_p^{0,1}(x,t) &= \begin{cases}
\hat{u}_p^{0,1}(x,t), & \text{if } \hat{u}_p^{0,1}(x,t) > 0, \\
\min\{\hat{\zeta}_p^{0,1}(x,t), 0\}, & \text{if } \hat{u}_p^{0,1}(x,t) = 0.
\end{cases}
\end{align*}
Then we compute the second local dual function $\zeta^{0,2}$ in the time segment by
\begin{equation*}
\zeta_c^{0,2} = \nabla z^0 \cdot \nabla y^0(u^{0,1})\quad\text{and}\quad
\zeta_p^{0,2} = z^0 y^0(u^{0,1}).
\end{equation*}
Note that $z^0$ is already computed. 
This enables computing $\eta^{0,2}$.
With the updated kernel, the action of $R$ differs from \eqref{eqn19}.
For example, the DFP correction reads
\begin{equation*}
\begin{aligned}
\eta_c^{0,2}(x,t) &= D_c(x,t) \zeta_c^{0,2}(x,t) + c_1 \hat{\eta}_c^{0,1} + c_2 D_c(x,t) \hat{\zeta}_c^{0,1}(x,t), \\
\eta_p^{0,2}(x,t) &= D_p(x,t) \zeta_v^{0,2}(x,t) + c_1 \hat{\eta}_p^{0,1} + c_2 D_p(x,t) \hat{\zeta}_p^{0,1}(x,t),
\end{aligned}
\end{equation*}
with
$c_1 = \frac{(\hat{\eta}^{0,1}, \zeta^{0,2})}{(\hat{\eta}^{0,1}, \hat{\zeta}^{0,1})}$ and $
c_2 = \frac{(D \hat{\zeta}^{0,1}, \zeta^{0,2})}{(D \hat{\zeta}^{0,1}, \hat{\zeta}^{0,1})}$.
The positivity constraint on $u^{0,2}$ is enforced by the operator $\mathcal{P}$, and the potential $y^{0}(u^{0,2})$ is obtained by solving \eqref{eqn1}.
The relative error $e$ is then recomputed, and the iteration continues until the relative errors falls below  $\epsilon_{\text{tol}}$.

\item \textbf{Step 23-26: Final estimation and transition to the next segment}

Once the relative error satisfies ${\|y^0(u^{0,M}) - y_{d}^0\|_{L^2}}/{\|y_{d}^0\|_{L^2}} \leq \epsilon_{\text{tol}}$ for {the $M$th inner iteration}, we compute the final estimate of $u_*^0$ using $\zeta^0$:
\begin{equation*}
\zeta_c^0 = \nabla z^0 \cdot \nabla y^0(u^{0,M})\quad\text{and}\quad
\zeta_p^0 = z^0 y^0(u^{0,M}).
\end{equation*}
The estimate of the inclusion is $u^0 = \mathcal{P}(\mathcal{R} \zeta^0)$, and we obtain the initial value for the next time segment by solving the following problem with Dirichlet boundary condition:
\begin{equation*}
\left\{
\begin{aligned}
\partial_{t} y_{D}^{0}(u^{0}) - \nabla \cdot \left((1 + u_{c}^{0}) \nabla y_{D}^{0}(u^{0})\right) + u_{p}^{0} y_{D}^{0}(u^{0}) &= f, & \text{in }& \Omega \times (0, \delta t), \\
{\partial_n} y_{D}^{0}(u^{0}) &= g, & \text{on }& \Gamma \times (0, \delta t), \\
y_{D}^{0}(u^{0}) &= h, & \text{in }& \Omega \times \{0\}.
\end{aligned}
\right.
\end{equation*}
Then a damping factor $\lambda$ is applied to the low-rank terms of $R$.
\end{itemize}

In subsequent time segments, the procedure is analogous, with $h$ replaced by $y_{D}^0(u^0)(\cdot, n\delta t)$ as the initial value.
Unlike the first iteration of the first segment, the kernel $R$ now includes both the delta distribution and the low-rank structure from the earlier segment.
\end{example}

\section{Numerical experiments} \label{sec:NUM}
\newcommand{\figlen}{0.16\textwidth}

Now we present numerical examples to illustrate the efficiency and robustness of the IDSM in Algorithm \ref{alg1}.
Example \ref{exam:1} involves two disconnected inclusions that first merge and then split.
Examples \ref{exam:2} and \ref{exam:3} consider the composite model \eqref{eqn6} and nonlinear model \eqref{eqn13}, respectively.
Examples \ref{exam:4} and \ref{exam:5} consider fading inclusions and diminishing inclusions to assess its ability to track dynamic behavior beyond spatial relocation.
The codes for reproducing the numerical results will be made available at the following github link \url{https://github.com/RaulWangfr/IDSM-parabolic.git}.

All the experiments are conducted on the unit disk $\Omega = \{(x_{1}, x_{2})|\,x_{1}^{2}+x_{2}^{2}< 1\} \subset \mathbb{R}^2$ and $T=10$.
We employ the Crank-Nicolson scheme and linear finite elements for time and spatial discretization, respectively.
The reference solution $y(u_*)$ is obtained by solving \eqref{eqn1}--\eqref{eqn:ibc} with a step size $\Delta t = 0.01$ on a fine mesh with 13870 triangular elements.
We generate the noisy data $y_d^\delta$ by adding noise into the exact data $y(u_{*})$ pointwise as
\begin{equation*}
y_d^\delta(x,t) = y(u_*)(x,t)\left(1 + \varepsilon \delta(x,t)\right),
\end{equation*}
where $\delta(x,t) \sim \mathcal{U}(-1,1)$ is uniformly distributed and $\varepsilon$ is the relative noise level.
The measurements are recorded at discrete times $\{0.01n\}_{0}^{1000}$, and estimated with linear interpolation at intermediate times.
All the examples employ the following source
\begin{align*}
f &= 25\sin\left(\tfrac{t\pi}{4}\right)\sin(3x)\cos(4y), \quad h = 3.0 + \sin(3x)\cos(4y),\\
g &= \cos\left(\tfrac{t\pi}{6}\right)\left(3\cos(3x)\cos(4y)\vec{n}_{x} - 4\sin(3x)\sin(4y)\vec{n}_{y}\right),
\end{align*}
Algorithm \ref{alg1} employs time segments of length $\delta t = 0.1$ with a fine time step $\Delta t = 0.0125$.
The fine mesh $\mathscr{T}_f$ and coarse mesh $\mathscr{T}_c$ comprise 7002 and 1120 triangular elements, respectively.
The kernel $R$ is initialized using a delta distribution with the boundary distance kernel $D(x,t) = d(x,\Gamma)^{1.4}\chi_{\Omega_{\epsilon}}(x)$ with $\Omega_{\epsilon}=\{x\in \Omega:\,d(x,\Gamma)\geq \epsilon\}$.
Unless otherwise specified, we set the noise level $\varepsilon$ to 5\%, the damping parameter $\lambda$ to 0.6 and the tolerance $\epsilon_{\text{tol}}$ to 0.08 (for the low-rank update).
While both DFP \eqref{eqn:DFP} and BFG \eqref{eqn:BFG} corrections are implemented, we present results using only one method for each example.
The reconstructions are visualized via heatmaps of the normalized inhomogeneity ${u^n}/{\|u^n\|_{L^\infty}}$, in which the presence of inclusions is represented by the color red, whereas the background is indicated by the color blue, and the boundaries of the ground truth inclusions $\omega_{j}$ are indicated by black contours.
Table \ref{table1} presents the overall computational expense, corresponding to $T/\delta t=100$ time segments.

\begin{table}[hbt!]
\centering
\begin{threeparttable}
\caption{The number of PDE solves required for the examples.
Each time segment involves solving \eqref{eqn5} and \eqref{eqn1} multiple times.
The reported values represent the average number of PDE solves computed over all time segments.}\label{table1}

\begin{tabular}{l|c c|c|c|c|c|c|}
\toprule
\textbf{Example} & \multicolumn{2}{c|}{1} & {2} & {3} & {4} & {5} \\
\cline{2-3}
& $\varepsilon = 5\%$ & $\varepsilon = 10\%$& &  & & \\
\hline
Forward background solves &1&1 & {1} & {1} & 1 & 1\\
Backward adjoint solves &1.01&1.11  & {1.01} & {2.73} & 1.02 & 1.01\\
Forward inhomogeneous solves &1.01&1.11  & {1.01} & {2.73} & 1.02 & 1.01 \\
Dirichlet verification solves &1&1  & {1} & {1} & 1 & 1\\
\hline
\textbf{Total solves} &4.02&4.22 & {4.02} & {7.46} & 4.04 & 4.02\\
\bottomrule
\end{tabular}
\end{threeparttable}
\end{table}

\subsection{Example 1: merging and splitting of inclusions}\label{exam:1}
This example investigates the reconstruction of two moving inclusions that merge into a single entity over time and then split, at two noise levels, $\varepsilon=5\%$ and $\varepsilon=10\%$, and the tolerance $\epsilon_{\text{tol}}$ is set to $10\%$.
The governing model is given by \eqref{eqn11}.
Inside the inhomogeneity the conductivity is $0.1$, and the pointwise projection is $\mathcal{P}(\eta)=\max\{\min\{\eta,0.0\},-0.99\}$.
Both inclusions are modeled as circular regions with a fixed radius of $0.2$.
Their center trajectories are given by
\begin{align*}
\vec{\gamma}_{1}(t) &= \left(0.6\cos\left(\tfrac{t\pi}{6}\right), -0.7\sin\left(\tfrac{t\pi}{6}\right)\right)\chi_{[0,3)}(t)+
\left(-0.6\cos\left(\tfrac{t\pi}{6}\right), -0.7\sin\left(\tfrac{t\pi}{6}\right)\right)\chi_{[3,10]}(t), \\
\vec{\gamma}_{2}(t) &= (-0.6\cos(\tfrac{t\pi}{6}), -0.7\sin(\tfrac{t\pi}{6}))\chi_{[0,6)}(t) + 
(-0.6\cos(\tfrac{(12 - t)\pi}{6}), -0.7\sin(\tfrac{(12 - t)\pi}{6}))\chi_{[6,10]}(t).
\end{align*}
The reconstructions by the IDSM using the BFG correction scheme are shown in Fig. \ref{fig1}.

The IDSM can reliably  track the dynamically merging and splitting inclusions for both noise levels.
At 5\% noise, the reconstructions (top two rows) closely match the exact trajectories and capture well topological changes. At 10\% noise, the accuracy (bottom two rows) degrades slightly, but the IDSM can still reliably identify the inclusion locations. The computational cost is low in both cases,  cf. Table \ref{table1}, showing the efficiency of the IDSM.

\begin{figure}[hbt!]
\centering
\begin{tabular}{ccccc}
\includegraphics[width = \figlen, trim = {0.3cm 0.1cm 0.3cm 0.1cm}, clip]{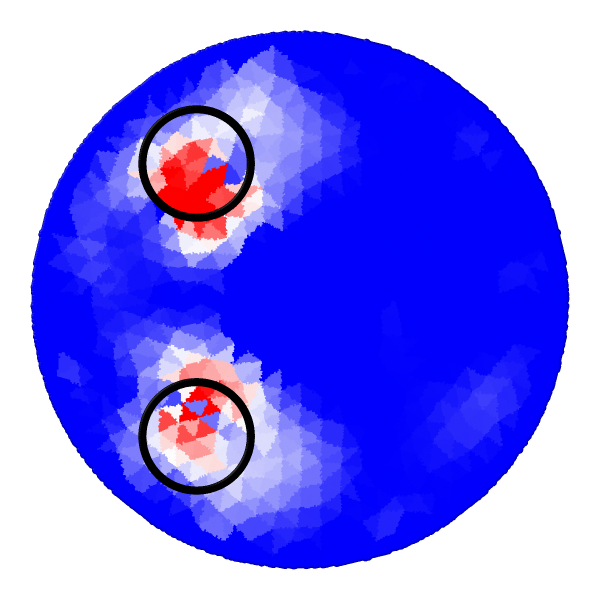}&
\includegraphics[width = \figlen, trim = {0.3cm 0.1cm 0.3cm 0.1cm}, clip]{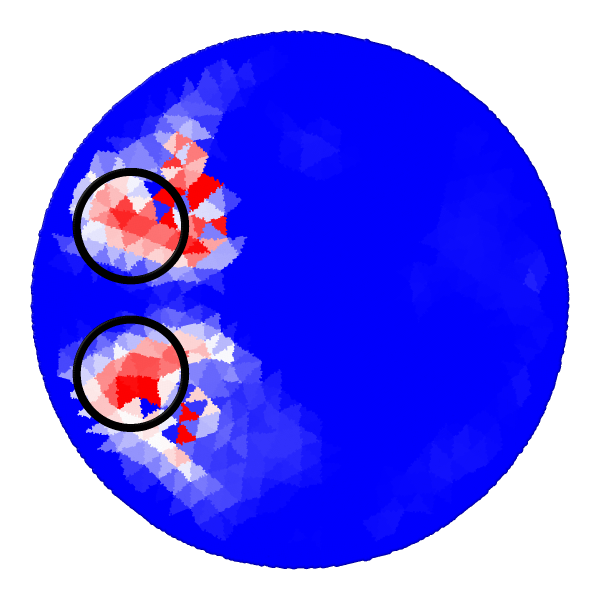}&
\includegraphics[width = \figlen, trim = {0.3cm 0.1cm 0.3cm 0.1cm}, clip]{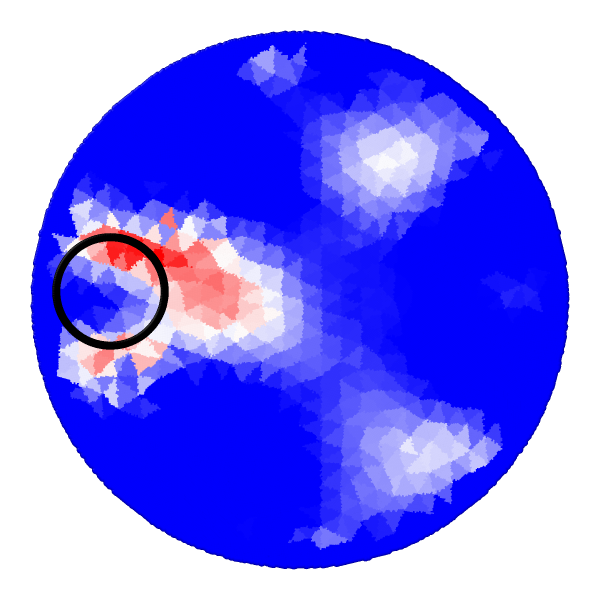}&
\includegraphics[width = \figlen, trim = {0.3cm 0.1cm 0.3cm 0.1cm}, clip]{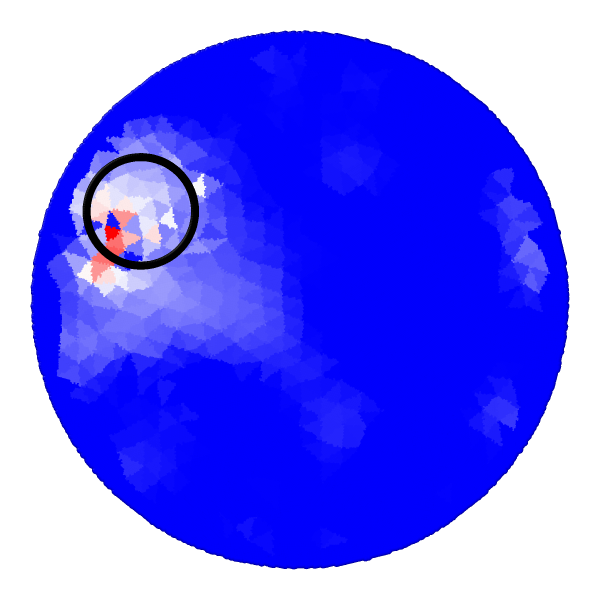}&
\includegraphics[width = \figlen, trim = {0.3cm 0.1cm 0.3cm 0.1cm}, clip]{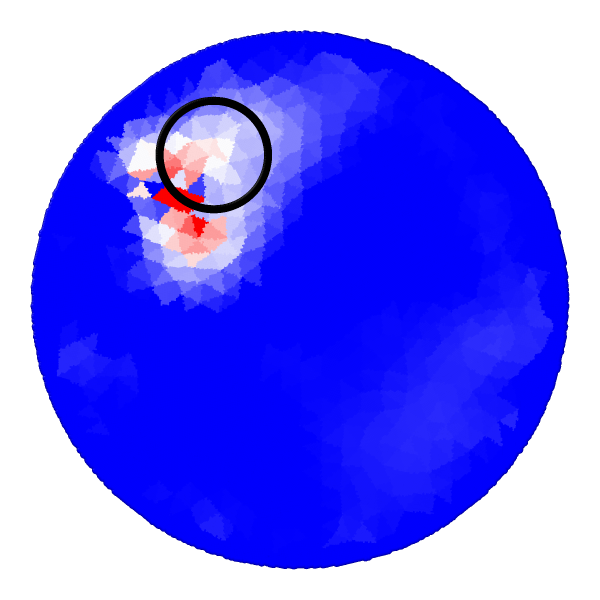}\\
$t=1$ & $t=2$ & $t=3$ & $t=4$ & $t=5$\\
\includegraphics[width = \figlen, trim = {0.3cm 0.1cm 0.3cm 0.1cm}, clip]{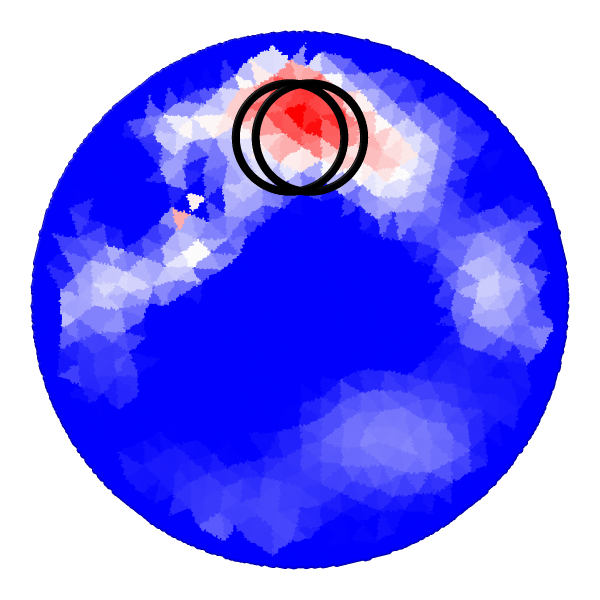}&
\includegraphics[width = \figlen, trim = {0.3cm 0.1cm 0.3cm 0.1cm}, clip]{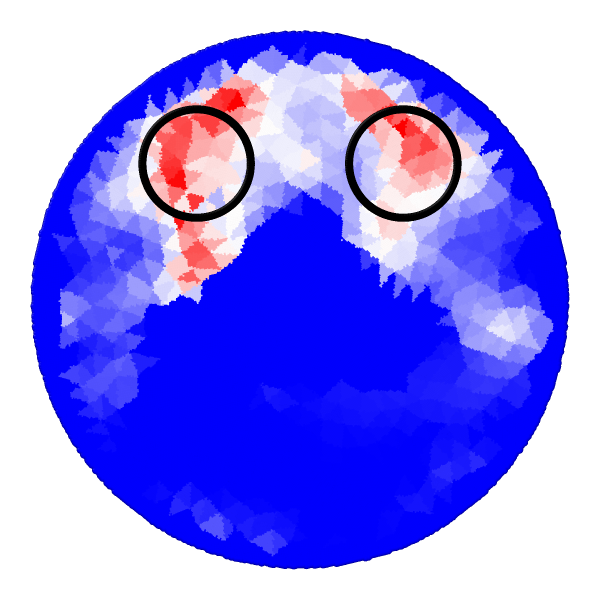}&
\includegraphics[width = \figlen, trim = {0.3cm 0.1cm 0.3cm 0.1cm}, clip]{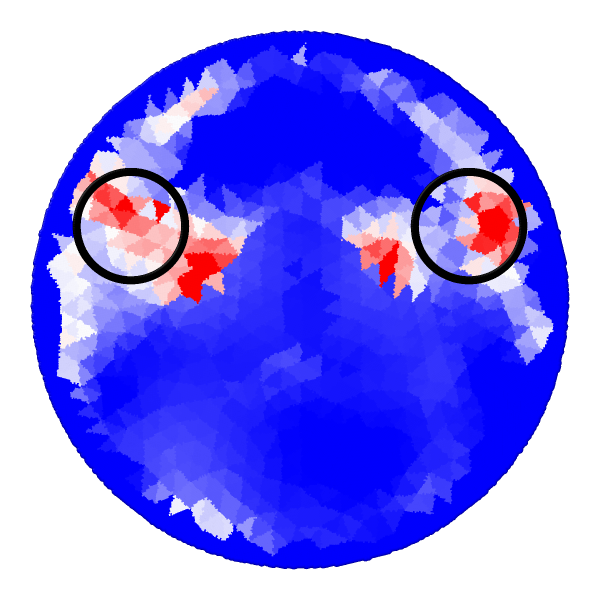}&
\includegraphics[width = \figlen, trim = {0.3cm 0.1cm 0.3cm 0.1cm}, clip]{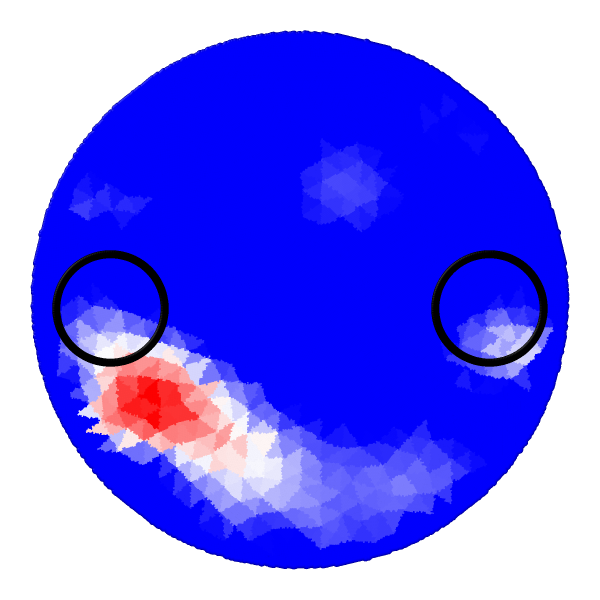}&
\includegraphics[width = \figlen, trim = {0.3cm 0.1cm 0.3cm 0.1cm}, clip]{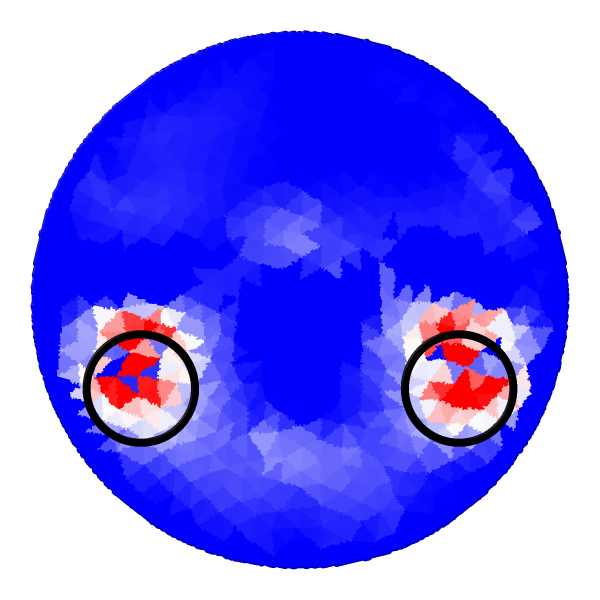}\\
$t=6$ & $t=7$ & $t=8$ & $t=9$ & $t=10$\\
\includegraphics[width = \figlen, trim = {0.3cm 0.1cm 0.3cm 0.1cm}, clip]{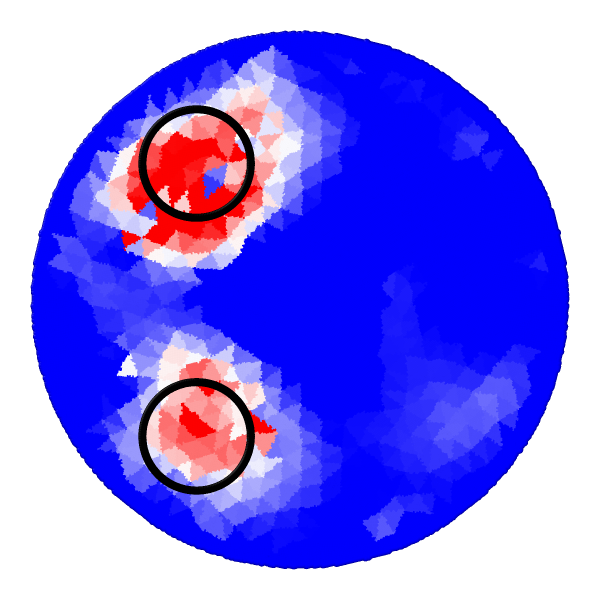}&
\includegraphics[width = \figlen, trim = {0.3cm 0.1cm 0.3cm 0.1cm}, clip]{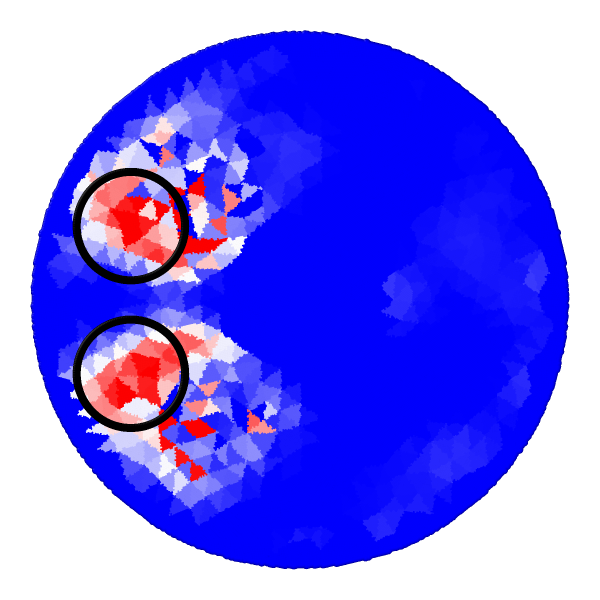}&
\includegraphics[width = \figlen, trim = {0.3cm 0.1cm 0.3cm 0.1cm}, clip]{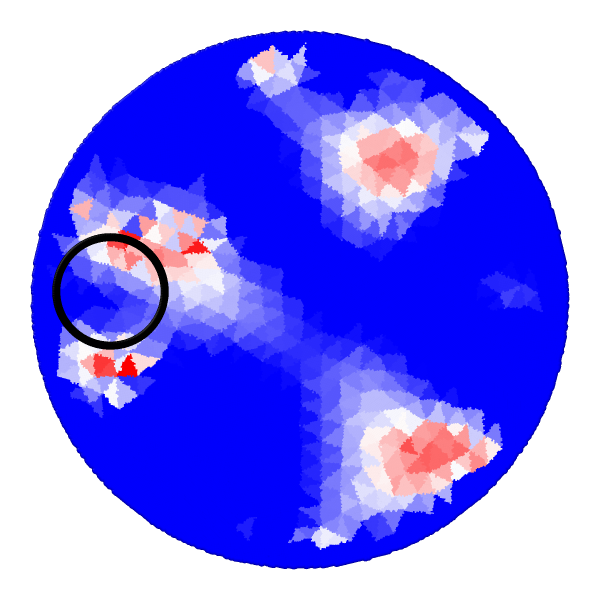}&
\includegraphics[width = \figlen, trim = {0.3cm 0.1cm 0.3cm 0.1cm}, clip]{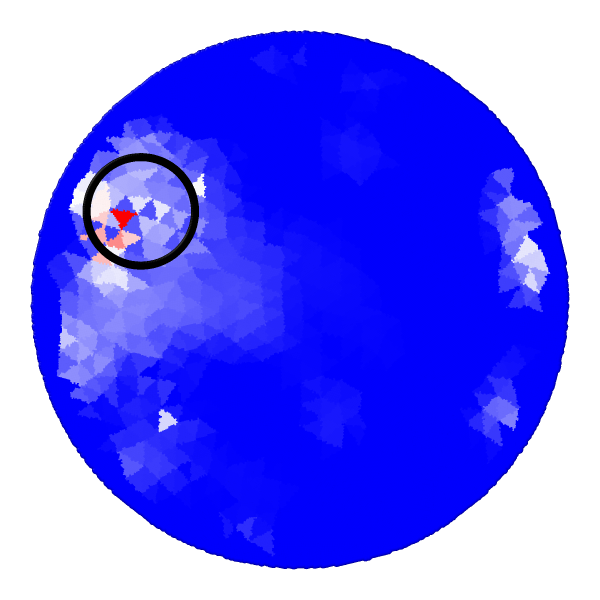}&
\includegraphics[width = \figlen, trim = {0.3cm 0.1cm 0.3cm 0.1cm}, clip]{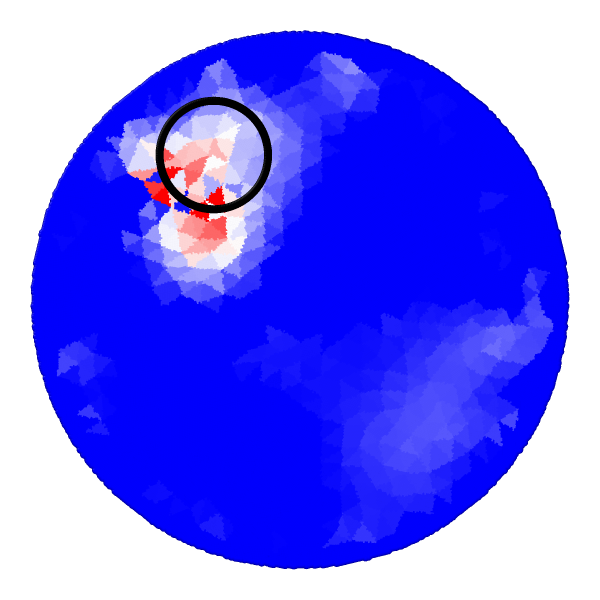}\\
$t=1$ & $t=2$ & $t=3$ & $t=4$ & $t=5$\\
\includegraphics[width = \figlen, trim = {0.3cm 0.1cm 0.3cm 0.1cm}, clip]{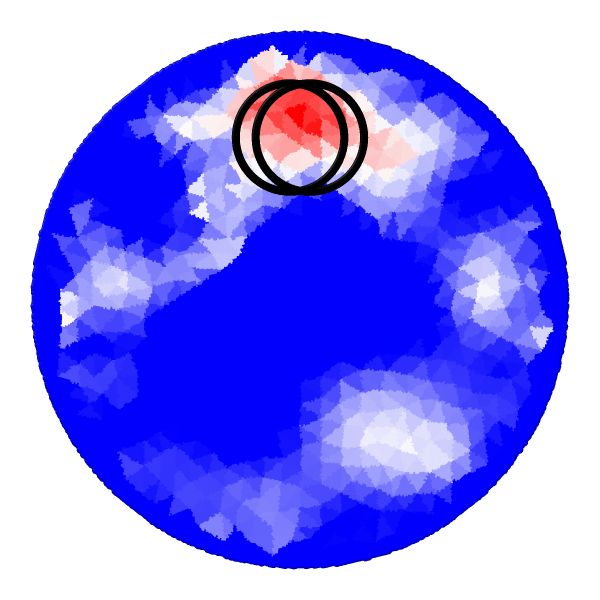}&
\includegraphics[width = \figlen, trim = {0.3cm 0.1cm 0.3cm 0.1cm}, clip]{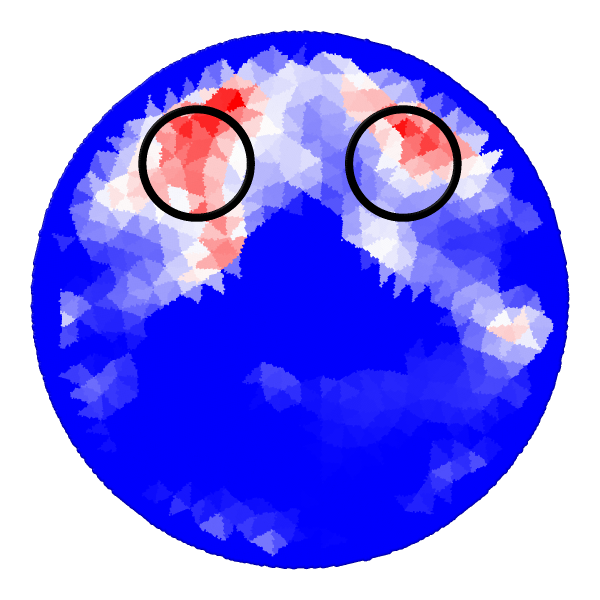}&
\includegraphics[width = \figlen, trim = {0.3cm 0.1cm 0.3cm 0.1cm}, clip]{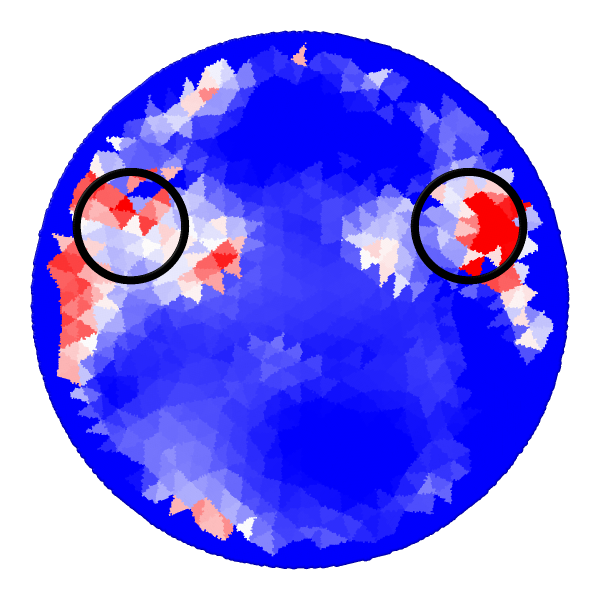}&
\includegraphics[width = \figlen, trim = {0.3cm 0.1cm 0.3cm 0.1cm}, clip]{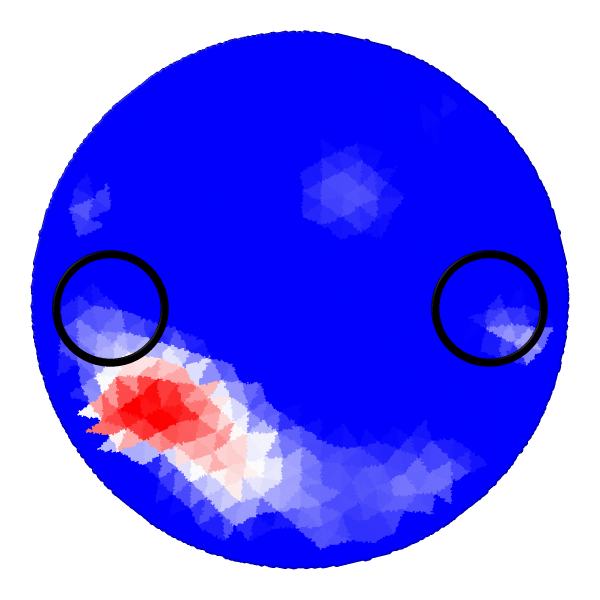}&
\includegraphics[width = \figlen, trim = {0.3cm 0.1cm 0.3cm 0.1cm}, clip]{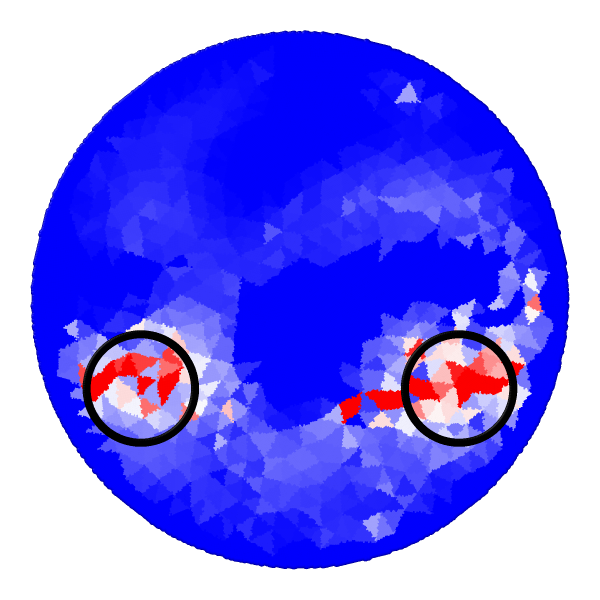}\\
$t=6$ & $t=7$ & $t=8$ & $t=9$ & $t=10$
\end{tabular}
\caption{
Numerical results for Example \ref{exam:1}, where two moving inclusions merge into one and then split.
Rows 1 and 2 are for $\varepsilon = 5\%$, while rows 3 and 4 are for $\varepsilon = 10\%$.
\label{fig1}}
\end{figure}

\subsection{Example 2: mixed-type inhomogeneities}\label{exam:2}
This example involves three moving inclusions with different types of inhomogeneity in the model \eqref{eqn6}: two conductivity inclusions and one potential inclusion.
We set $u_p = 15$ and $u_c = -0.9$: inside the conductivity inclusions, the conductivity is $0.1$, and inside the potential inclusion, the potential is $15$.
In the background, the conductivity and potential are $1$ and $0$, respectively. All three inclusions are circular with a radius of $0.2$, and their center trajectories are given by
\begin{align*}
\vec{\gamma}_{c1}(t) &= \left(0.65\cos\left(\tfrac{t\pi}{8} - \tfrac{7\pi}{6}\right), 0.65\sin\left(\tfrac{t\pi}{8} - \tfrac{7\pi}{6}\right)\right), \\
\vec{\gamma}_{c2}(t) &= \left(0.6\cos\left(\tfrac{t\pi}{8} - \tfrac{\pi}{3}\right), 0.7\sin\left(\tfrac{t\pi}{8} - \tfrac{\pi}{3}\right)\right), \\
\vec{\gamma}_{p}(t) &= \left(0.7\cos\left(\tfrac{t\pi}{8} - \tfrac{\pi}{3}\right), 0.5\sin\left(\tfrac{t\pi}{8} - \tfrac{\pi}{3}\right)\right).
\end{align*}
These trajectories are chosen to test the ability of the IDSM to track dynamic inclusions with overlapping.
We take
$\mathcal{P}(\eta_c)=\max\{\min\{\eta_{c},0.0\},-0.99\}$ and $\mathcal{P}(\eta_{p})
\min\{\max\{\eta_{p},0.0\},30\}$. The results by the IDSM using the DFP correction scheme are shown in Fig. \ref{fig2}.

The proposed IDSM successfully distinguishes the spatial and temporal distributions of the mixed inhomogeneities, cf. Fig. \ref{fig2}.
It can accurately localize the conductivity and potential inclusions even in the presence of near-overlap.
This shows the robustness of the IDSM in handling inhomogeneities of distinct physical origins.

\begin{figure}[hbt!]
\centering
\begin{tabular}{ccccc}
\includegraphics[width = \figlen, trim = {0.3cm 0.1cm 0.3cm 0.1cm}, clip]{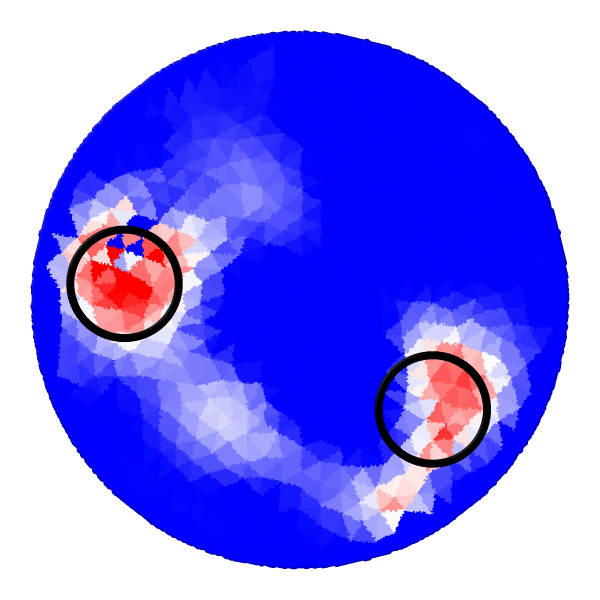}&
\includegraphics[width = \figlen, trim = {0.3cm 0.1cm 0.3cm 0.1cm}, clip]{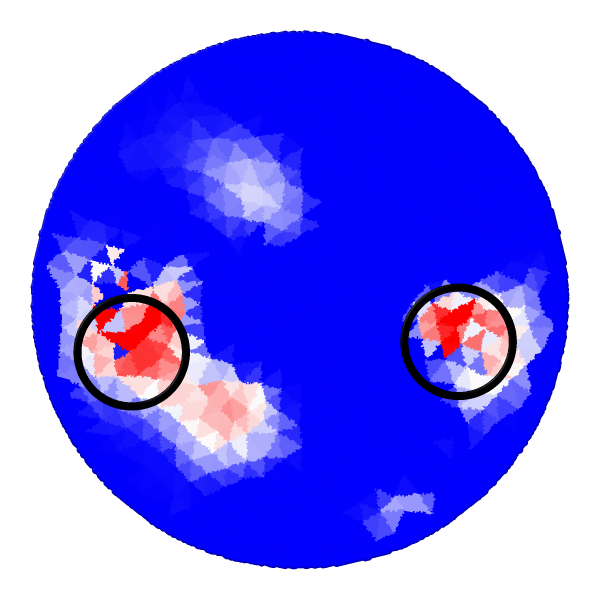}&
\includegraphics[width = \figlen, trim = {0.3cm 0.1cm 0.3cm 0.1cm}, clip]{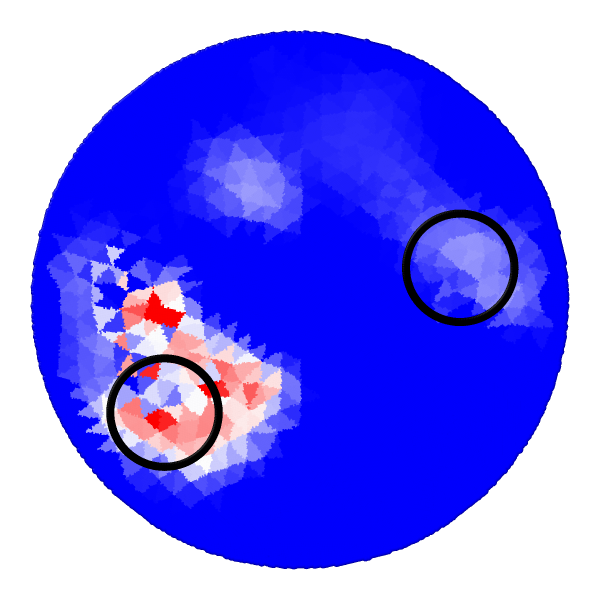}&
\includegraphics[width = \figlen, trim = {0.3cm 0.1cm 0.3cm 0.1cm}, clip]{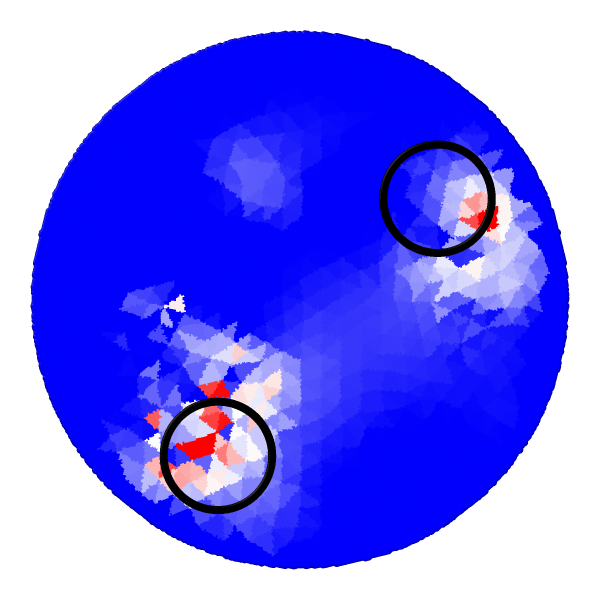}&
\includegraphics[width = \figlen, trim = {0.3cm 0.1cm 0.3cm 0.1cm}, clip]{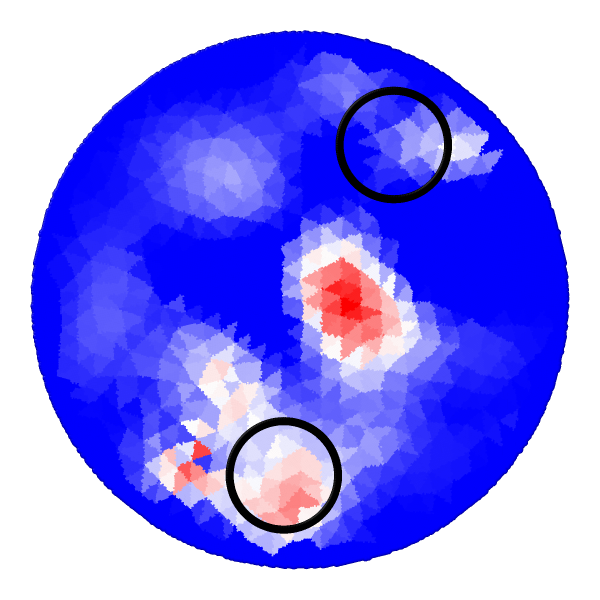}\\
$t=1$ & $t=2$ & $t=3$ & $t=4$ & $t=5$\\
\includegraphics[width = \figlen, trim = {0.3cm 0.1cm 0.3cm 0.1cm}, clip]{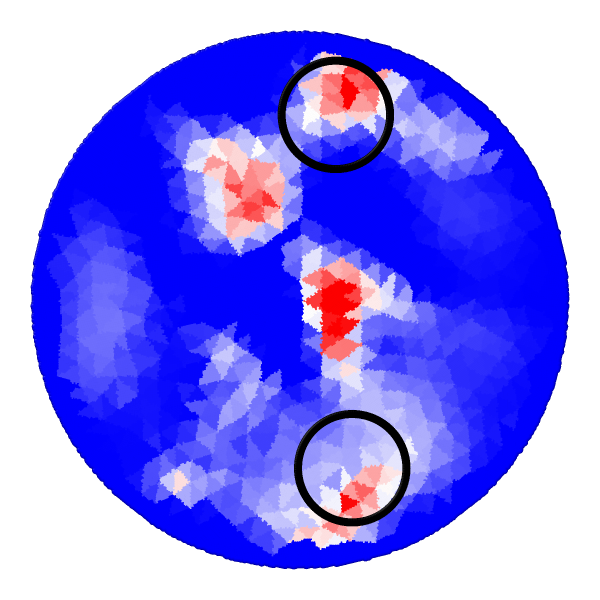}&
\includegraphics[width = \figlen, trim = {0.3cm 0.1cm 0.3cm 0.1cm}, clip]{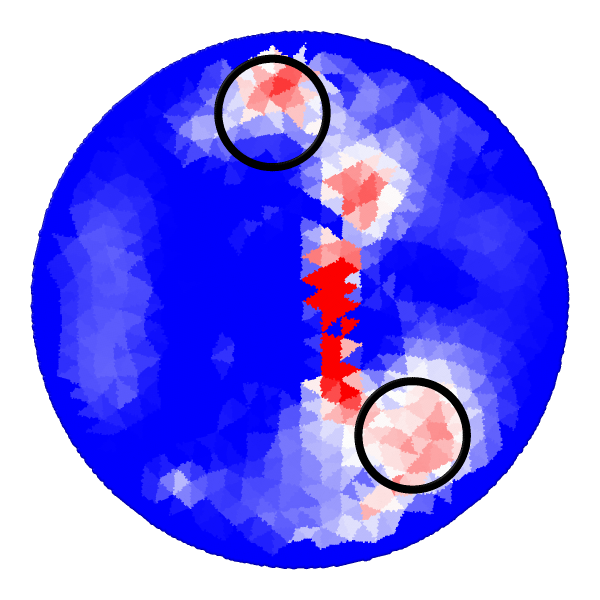}&
\includegraphics[width = \figlen, trim = {0.3cm 0.1cm 0.3cm 0.1cm}, clip]{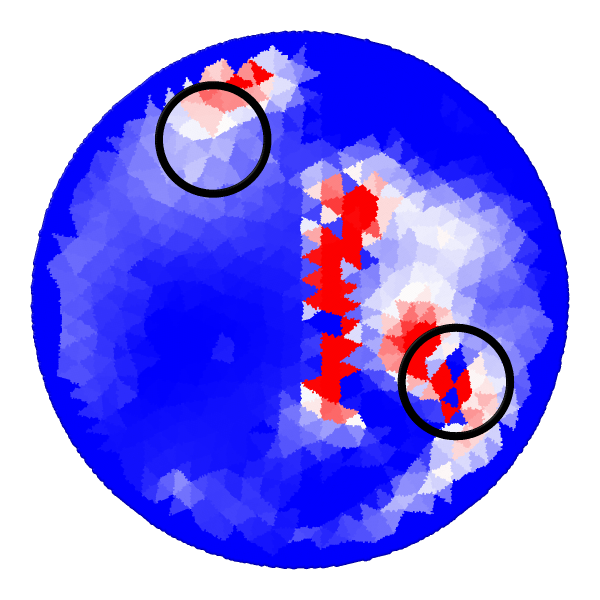}&
\includegraphics[width = \figlen, trim = {0.3cm 0.1cm 0.3cm 0.1cm}, clip]{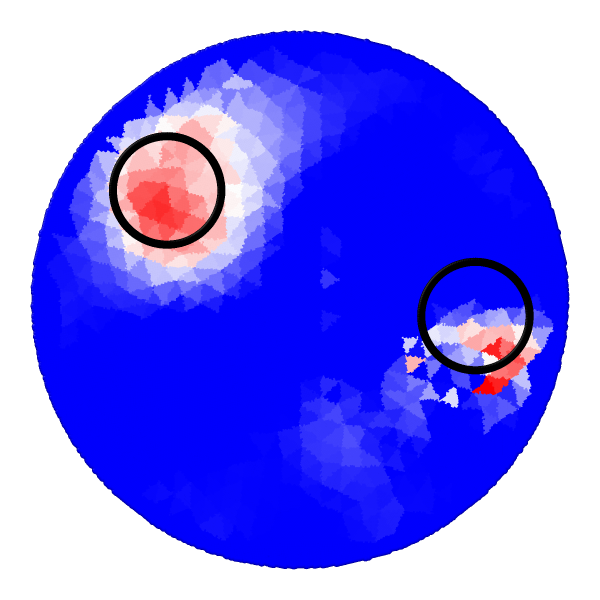}&
\includegraphics[width = \figlen, trim = {0.3cm 0.1cm 0.3cm 0.1cm}, clip]{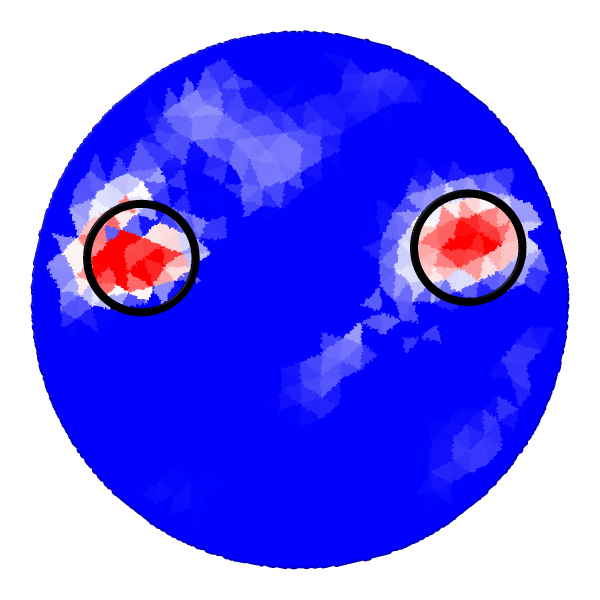}\\
$t=6$ & $t=7$ & $t=8$ & $t=9$ & $t=10$\\
\includegraphics[width = \figlen, trim = {0.3cm 0.1cm 0.3cm 0.1cm}, clip]{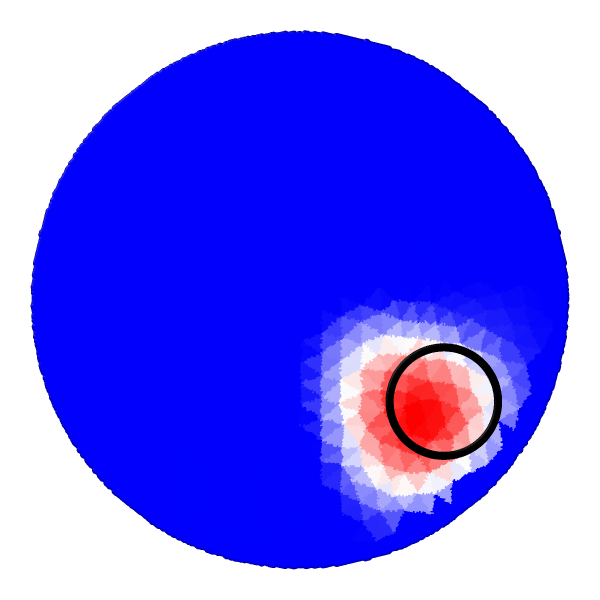}&
\includegraphics[width = \figlen, trim = {0.3cm 0.1cm 0.3cm 0.1cm}, clip]{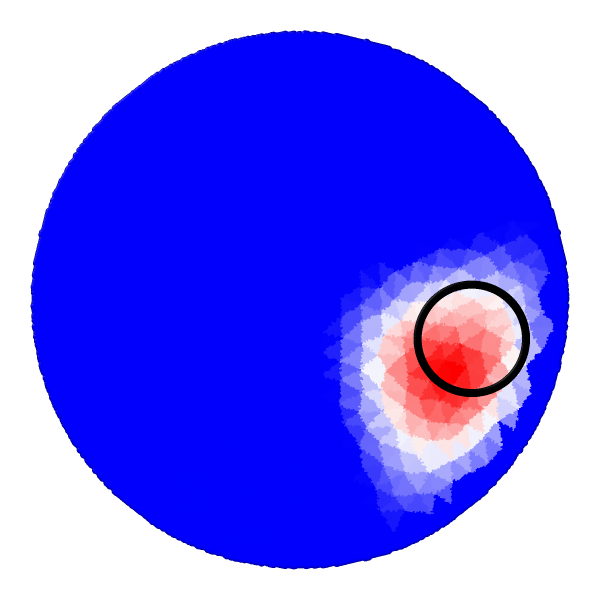}&
\includegraphics[width = \figlen, trim = {0.3cm 0.1cm 0.3cm 0.1cm}, clip]{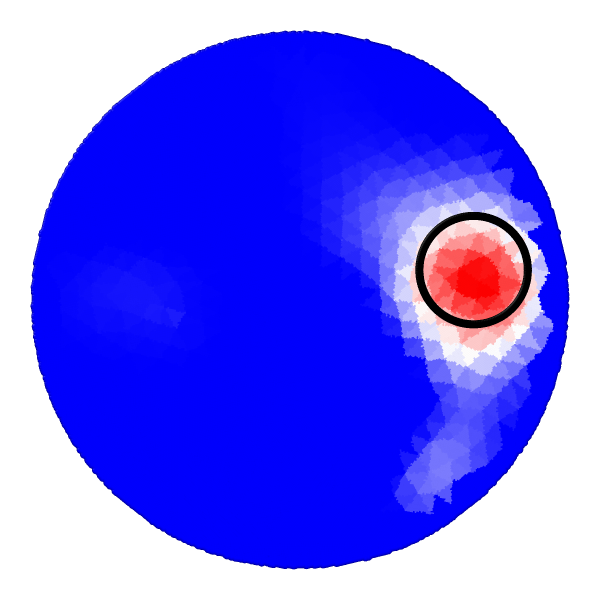}&
\includegraphics[width = \figlen, trim = {0.3cm 0.1cm 0.3cm 0.1cm}, clip]{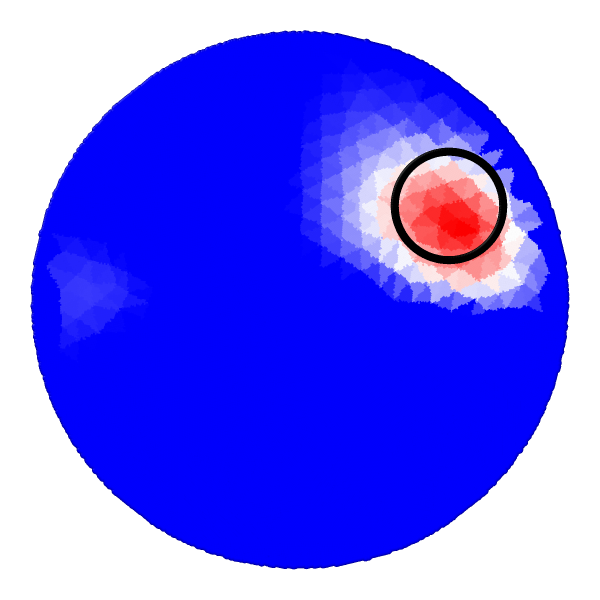}&
\includegraphics[width = \figlen, trim = {0.3cm 0.1cm 0.3cm 0.1cm}, clip]{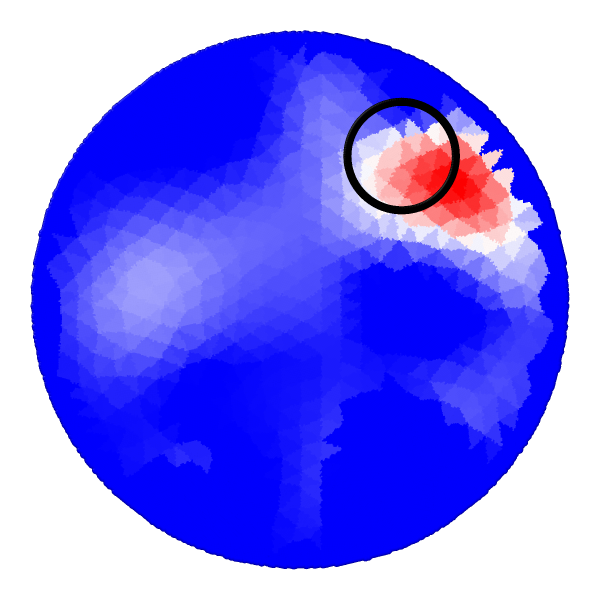}\\
$t=1$ & $t=2$ & $t=3$ & $t=4$ & $t=5$\\
\includegraphics[width = \figlen, trim = {0.3cm 0.1cm 0.3cm 0.1cm}, clip]{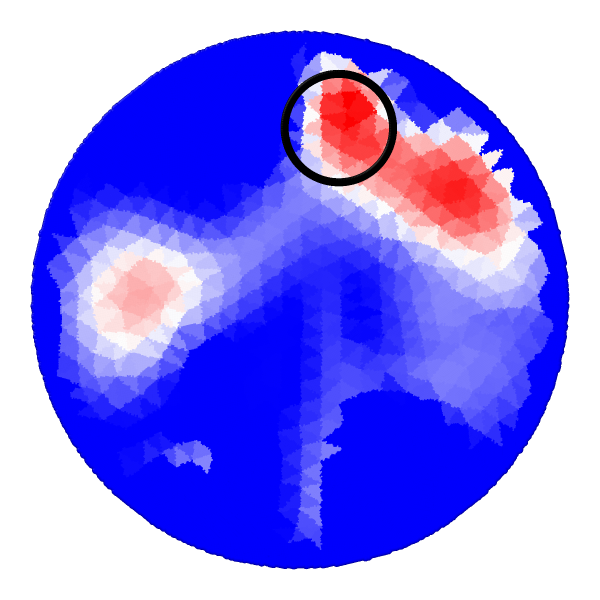}&
\includegraphics[width = \figlen, trim = {0.3cm 0.1cm 0.3cm 0.1cm}, clip]{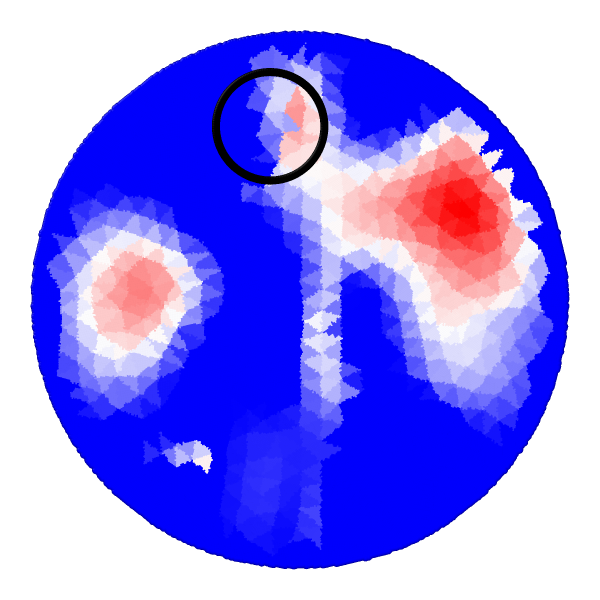}&
\includegraphics[width = \figlen, trim = {0.3cm 0.1cm 0.3cm 0.1cm}, clip]{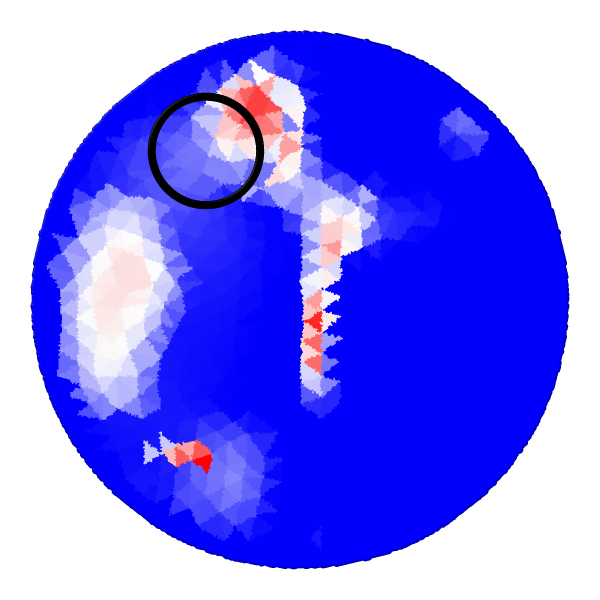}&
\includegraphics[width = \figlen, trim = {0.3cm 0.1cm 0.3cm 0.1cm}, clip]{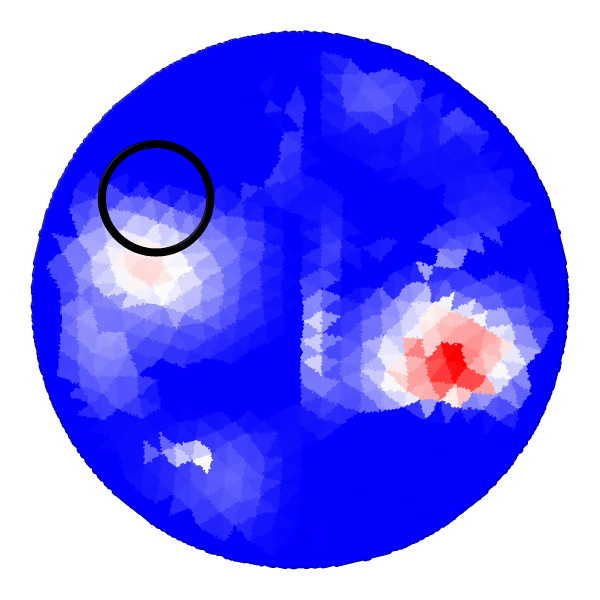}&
\includegraphics[width = \figlen, trim = {0.3cm 0.1cm 0.3cm 0.1cm}, clip]{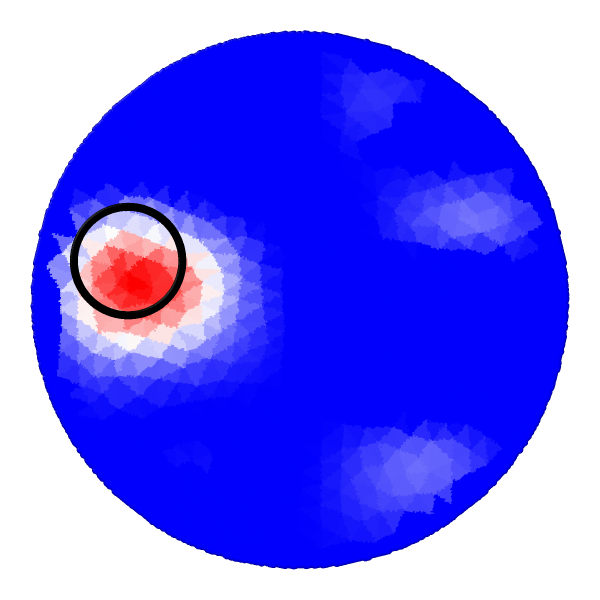}\\
$t=6$ & $t=7$ & $t=8$ & $t=9$ & $t=10$
\end{tabular}
\caption{
Numerical results for Example \ref{exam:2}, which involves mixed-type inhomogeneities.
The top and bottom two rows are for the conductivity and potential inclusions, respectively.\label{fig2}}
\end{figure}

\subsection{Example 3: nonlinear type}\label{exam:3}
This example considers a nonlinear potential problem described by \eqref{eqn13} with $p=3$, where $N(y)u=u|y|y$, and only one pair of Cauchy data is utilized.
By Sobolev embedding theorem \cite{Evans2010}, $H^{1}(\Omega)\subset L^{p}(\Omega)$ for any $p\geq 1$ in $\mathbb{R}^{2}$, ensuring the well-definedness of the operator.
Furthermore, the operator is monotone:
\begin{equation*}
\langle N(y_{1})u-N(y_{2})u, y_{1}-y_{2}\rangle_{(H^{1}(\Omega))'\times H^{1}(\Omega)}=\int_{\Omega}u(y_{1}-y_{2})(|y_{1}|y_{1}-|y_{2}|y_{2})\rmd x>0.
\end{equation*}
Thus the parabolic problem is well-posed.
$\mathcal{P}$ is set to $\mathcal{P}(\eta)=\min\{\max\{\eta, 0.0\}, 40\}$, and $u$ is set to $20$ inside a circular inclusion of radius $0.2$, with the trajectory
\begin{equation*}
\vec{\gamma}_{1}(t)=(0.5\cos(t\pi/ 6+\pi/4),0.7\sin(t\pi/ 6+\pi/4)).
\end{equation*}
The reconstructions by the IDSM using the BFG correction scheme are shown in Fig. \ref{fig3}.

The results show that the method can accurately capture the inclusion trajectory of the nonlinear inhomogeneity with only one pair of Cauchy data.
The low-rank correction dynamically adapts to the nonlinear potential term $N(y)$.
Table \ref{table1} shows that the computational expense for the model using the IDSM is nearly identical to that for the linear case.
Thus, the method can be extended to complex scenarios without sacrificing efficiency.
The nonlinear term $N(y)u=u|y|y$ is nonsmooth in $y$, presenting significant challenges for traditional regularization methods, whereas the IDSM handles the term easily without requiring additional adaptations.

\begin{figure}[hbt!]
\centering
\begin{tabular}{ccccc}
\includegraphics[width = \figlen, trim = {0.3cm 0.1cm 0.3cm 0.1cm}, clip]{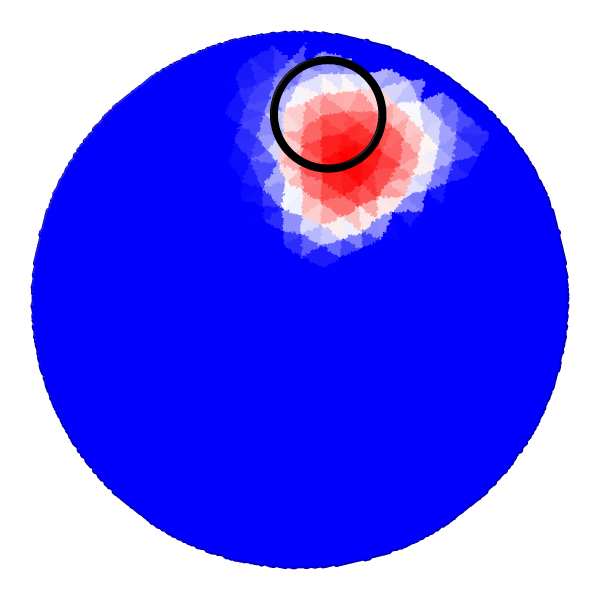}&
\includegraphics[width = \figlen, trim = {0.3cm 0.1cm 0.3cm 0.1cm}, clip]{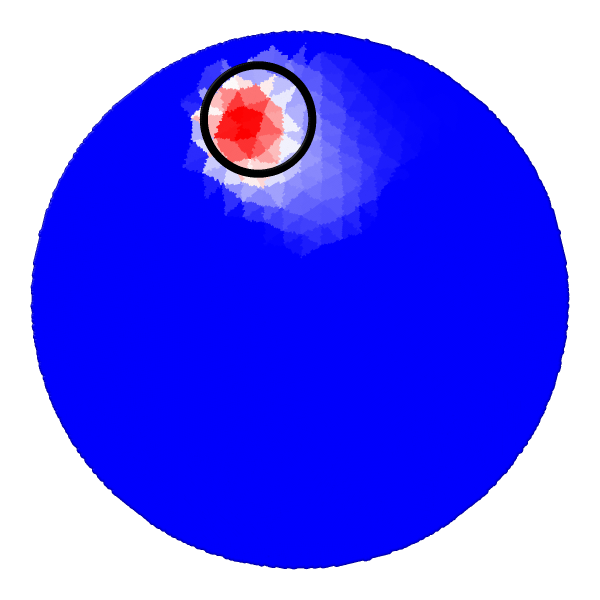}&
\includegraphics[width = \figlen, trim = {0.3cm 0.1cm 0.3cm 0.1cm}, clip]{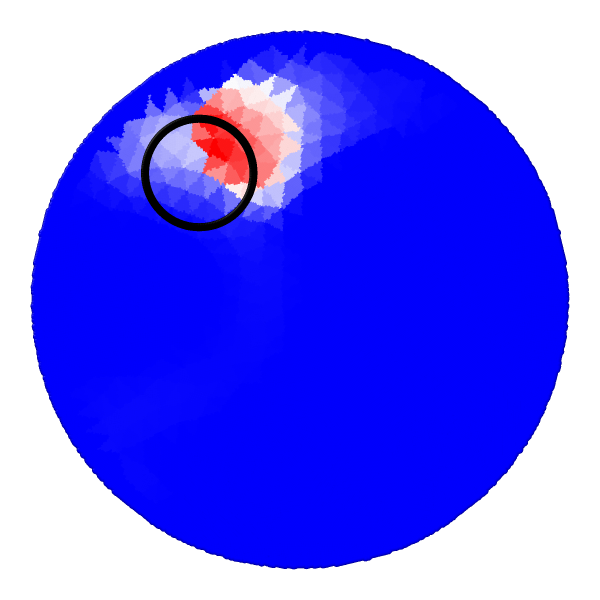}&
\includegraphics[width = \figlen, trim = {0.3cm 0.1cm 0.3cm 0.1cm}, clip]{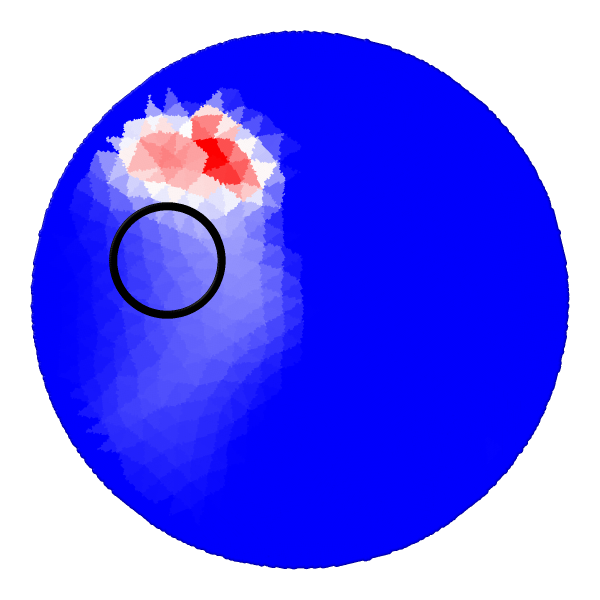}&
\includegraphics[width = \figlen, trim = {0.3cm 0.1cm 0.3cm 0.1cm}, clip]{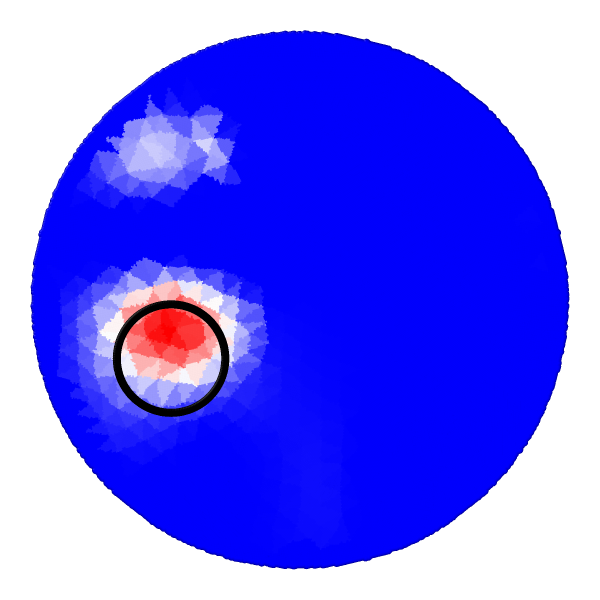}\\
$t=1$ & $t=2$ & $t=3$ & $t=4$ & $t=5$\\
\includegraphics[width = \figlen, trim = {0.3cm 0.1cm 0.3cm 0.1cm}, clip]{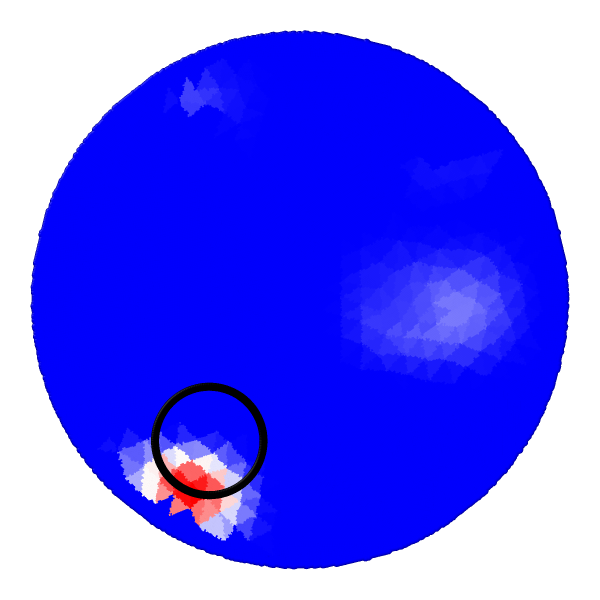}&
\includegraphics[width = \figlen, trim = {0.3cm 0.1cm 0.3cm 0.1cm}, clip]{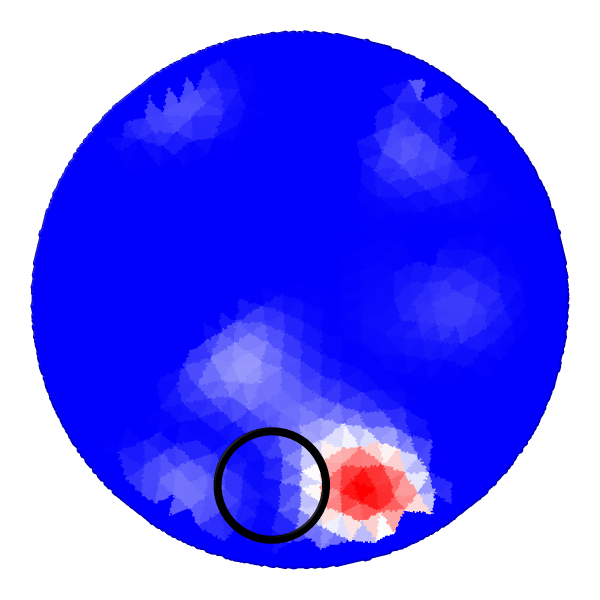}&
\includegraphics[width = \figlen, trim = {0.3cm 0.1cm 0.3cm 0.1cm}, clip]{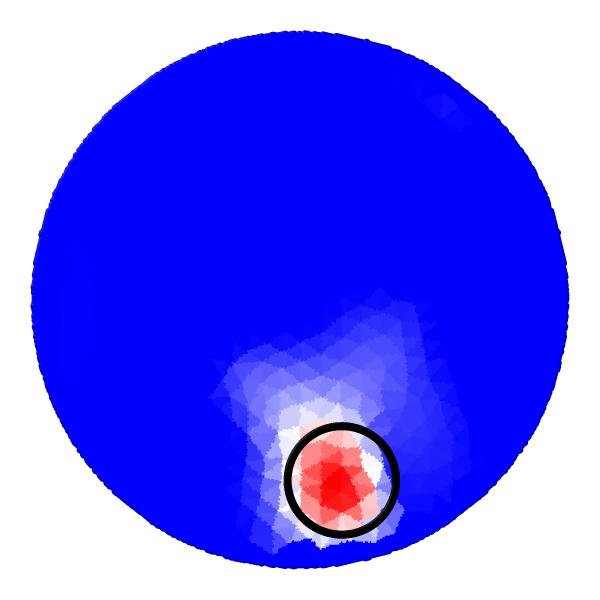}&
\includegraphics[width = \figlen, trim = {0.3cm 0.1cm 0.3cm 0.1cm}, clip]{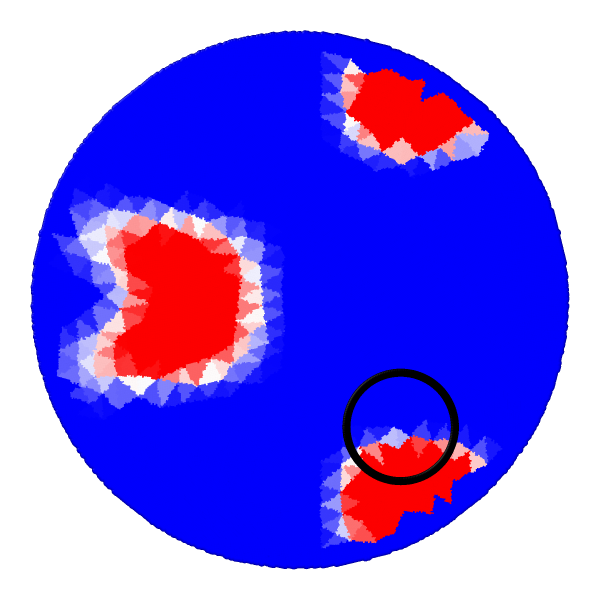}&
\includegraphics[width = \figlen, trim = {0.3cm 0.1cm 0.3cm 0.1cm}, clip]{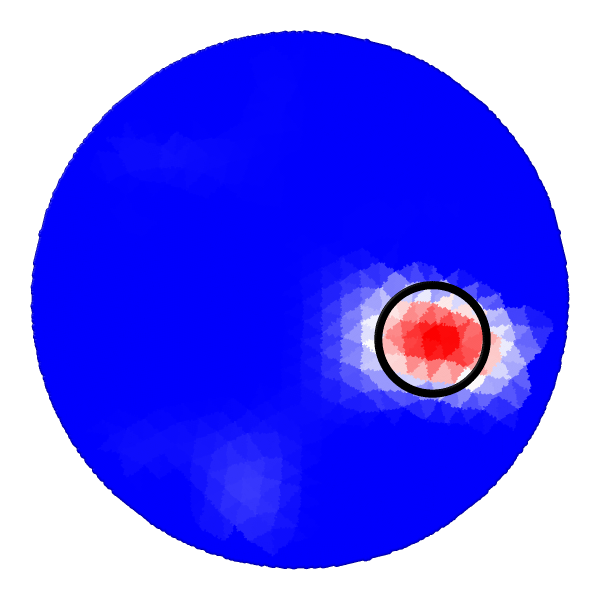}\\
$t=6$ & $t=7$ & $t=8$ & $t=9$ & $t=10$
\end{tabular}
\caption{
Numerical results for Example \ref{exam:3}, which involves nonlinear inhomogeneities.\label{fig3}}
\end{figure}

\subsection{Example 4: Fading potential inhomogeneities}\label{exam:4}
This example focuses on identifying potential inhomogeneities in the parabolic problem
\begin{equation}\label{eqn12}
\partial_{t} y - \Delta y + u_{p} y = f, \quad \text{in } \Omega \times (0, T), 
\end{equation}
with \eqref{eqn:ibc}, where $u_{p}$ represents the spatial-temporal distribution of potential inhomogeneity.
Consider two circular inclusions, each with a radius of $0.2$, whose centers follow the trajectories
\begin{align*}
\vec{\gamma}_{1}(t) &= \left(0.7\cos\left(\tfrac{t\pi}{8} \right), 0.6\sin\left(\tfrac{t\pi}{8} \right)\right), \\
\vec{\gamma}_{2}(t) &= \left(0.5\cos\left(\tfrac{t\pi}{8} + \tfrac{4\pi}{5}\right), 0.6\cos\left(\tfrac{t\pi}{8} + \tfrac{4\pi}{5}\right)\right).
\end{align*}
The potential within the inclusions depends linearly with time: $p_{1}(t) = \max\{15 - 2.5t, 0.0\}$ and $p_{2}(t) = \min\{2.5t, 15\}$, respectively, with the first and second inclusions fading and enhancing, respectively.
The pointwise projection $\mathcal{P}$ is set to $\mathcal{P}(\eta)=\min\{\max\{\eta, 0.0\},30\}$.
The reconstructions by the IDSM using the DFP correction scheme are shown in Fig. \ref{fig4}.

The results indicate that the IDSM accurately tracks the trajectories of both inclusions with only one pair of lateral Cauchy data.
While the coefficients are not resolved, the IDSM effectively captures the fading and enhancing behavior of the two inclusions.
Thus, the IDSM is capable of jointly recovering the spatial dynamics and the coefficients change of the evolving inhomogeneities.

\begin{figure}[hbt!]
\centering
\begin{tabular}{ccccc}
\includegraphics[width = \figlen, trim = {0.3cm 0.1cm 0.3cm 0.1cm}, clip]{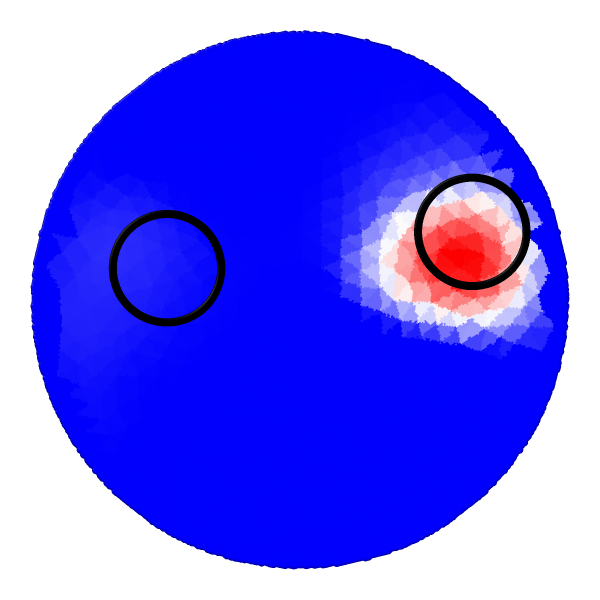}&
\includegraphics[width = \figlen, trim = {0.3cm 0.1cm 0.3cm 0.1cm}, clip]{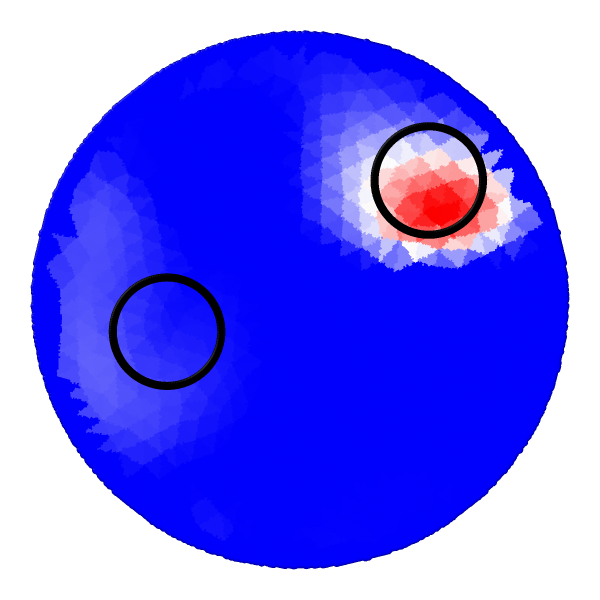}&
\includegraphics[width = \figlen, trim = {0.3cm 0.1cm 0.3cm 0.1cm}, clip]{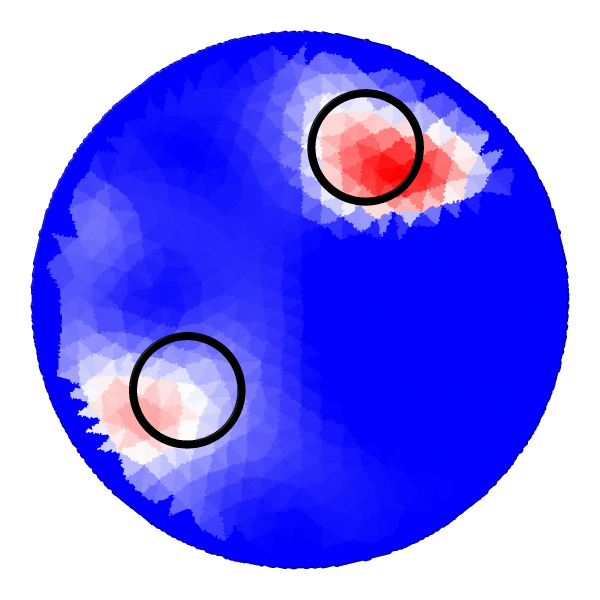}&
\includegraphics[width = \figlen, trim = {0.3cm 0.1cm 0.3cm 0.1cm}, clip]{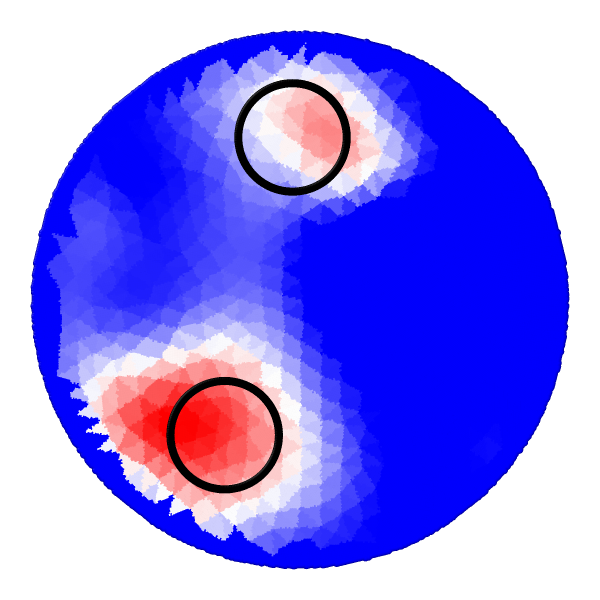}&
\includegraphics[width = \figlen, trim = {0.3cm 0.1cm 0.3cm 0.1cm}, clip]{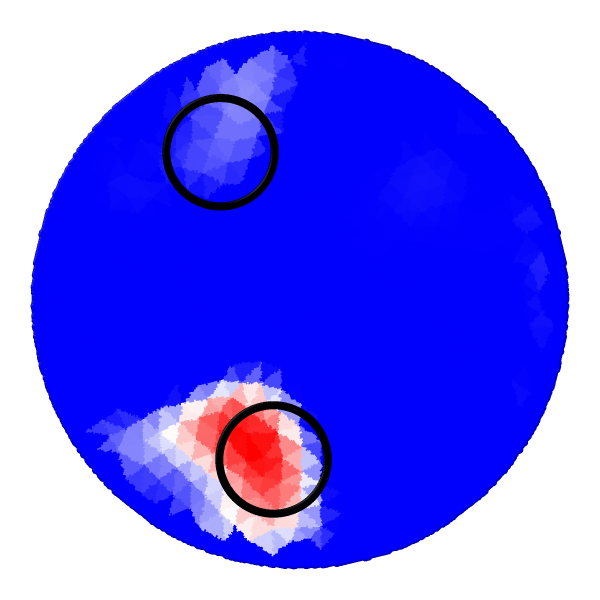}\\
$t=1$ & $t=2$ & $t=3$ & $t=4$ & $t=5$\\
\includegraphics[width = \figlen, trim = {0.3cm 0.1cm 0.3cm 0.1cm}, clip]{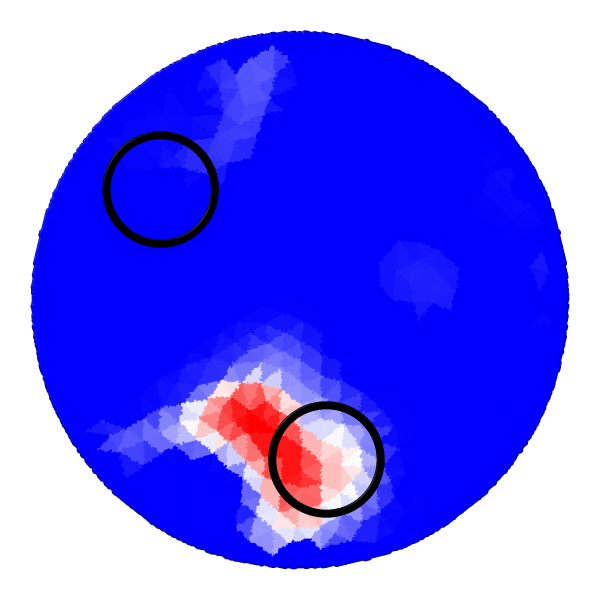}&
\includegraphics[width = \figlen, trim = {0.3cm 0.1cm 0.3cm 0.1cm}, clip]{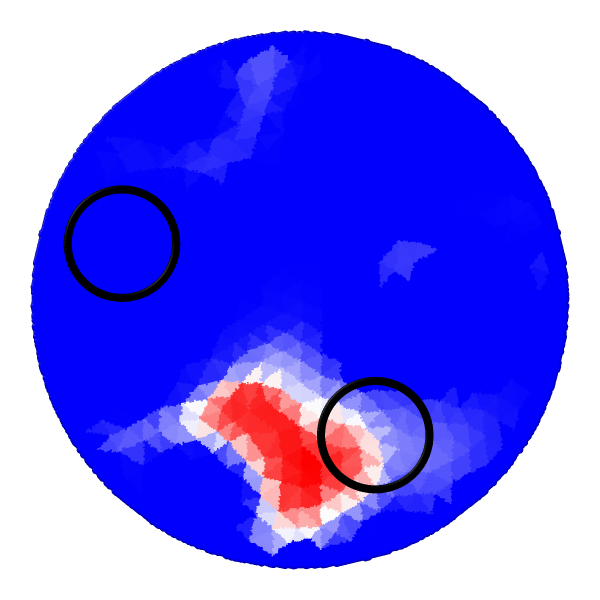}&
\includegraphics[width = \figlen, trim = {0.3cm 0.1cm 0.3cm 0.1cm}, clip]{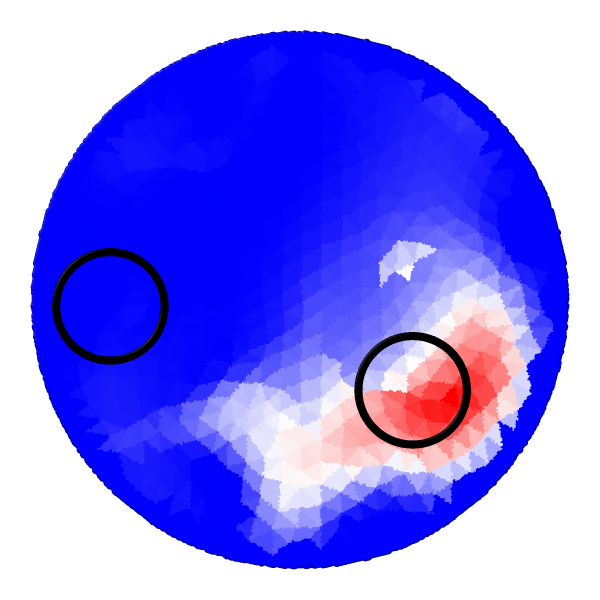}&
\includegraphics[width = \figlen, trim = {0.3cm 0.1cm 0.3cm 0.1cm}, clip]{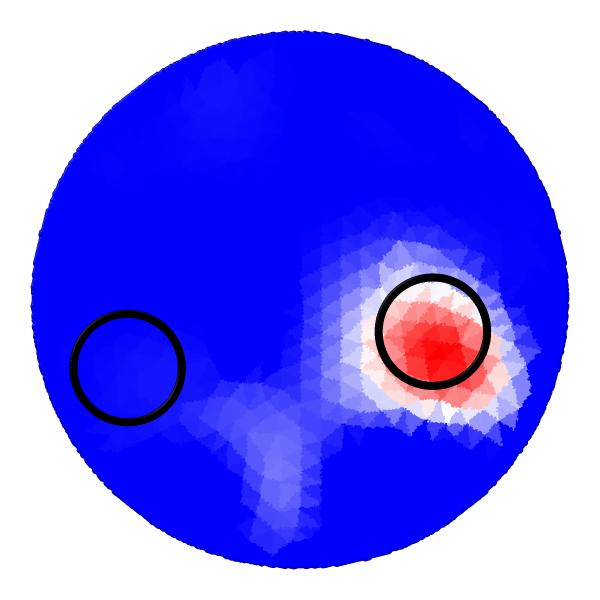}&
\includegraphics[width = \figlen, trim = {0.3cm 0.1cm 0.3cm 0.1cm}, clip]{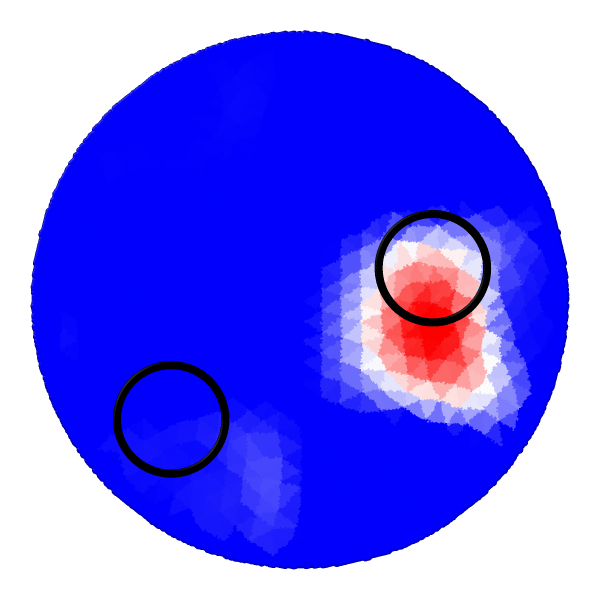}\\
$t=6$ & $t=7$ & $t=8$ & $t=9$ & $t=10$
\end{tabular}
\caption{Numerical results for Example \ref{exam:4}, which involves fading inhomogeneities.\label{fig4}}
\end{figure}

\subsection{Example 5: diminishing inhomogeneities}\label{exam:5}
This example extends the setting of Example \ref{exam:1}, i.e., \eqref{eqn11} with the same parameter setting, but involves a more intricate scenario where one of the inclusions diminishes over time, and utilizes only one pair of Cauchy data.
Both inclusions are circular, with their radii and center trajectories given by
\begin{equation*}
\left\{
\begin{aligned}
r_{1}(t) &= 0.2,&
\vec{\gamma}_{1}(t) &= \left(0.7\cos\left(\tfrac{t\pi}{6} + \tfrac{\pi}{3}\right), 0.6\sin\left(\tfrac{t\pi}{6} + \tfrac{\pi}{3}\right)\right), \\
r_{2}(t) &= \max(0.3 - 0.03t, 0.0),&
\vec{\gamma}_{2}(t) &= \left(0.6\cos\left(\tfrac{t\pi}{6} - \tfrac{2\pi}{3}\right), 0.5\sin\left(\tfrac{t\pi}{6} - \tfrac{2\pi}{3}\right)\right).
\end{aligned}
\right.
\end{equation*}
The first inclusion maintains a radius of $0.2$, while the second gradually shrinks over time and finally vanishes completely.
The reconstructions by the IDSM using the BFG correction scheme are shown in Fig. \ref{fig5}.

The results show that the IDSM captures the dynamics of the diminishing inclusion with limited data.
This example and Examples \ref{exam:3} and \ref{exam:4} indicate that the IDSM requires only one single pair of data if only one type of inhomogeneity is involved, even the dynamic behavior of the inclusion is complex or the inhomogeneity is nonlinear.
The algorithm is computationally efficient (about 4-5 direct PDE solves), cf. Table \ref{table1}.

\begin{figure}[hbt!]
\centering
\begin{tabular}{ccccc}
\includegraphics[width = \figlen, trim = {0.3cm 0.1cm 0.3cm 0.1cm}, clip]{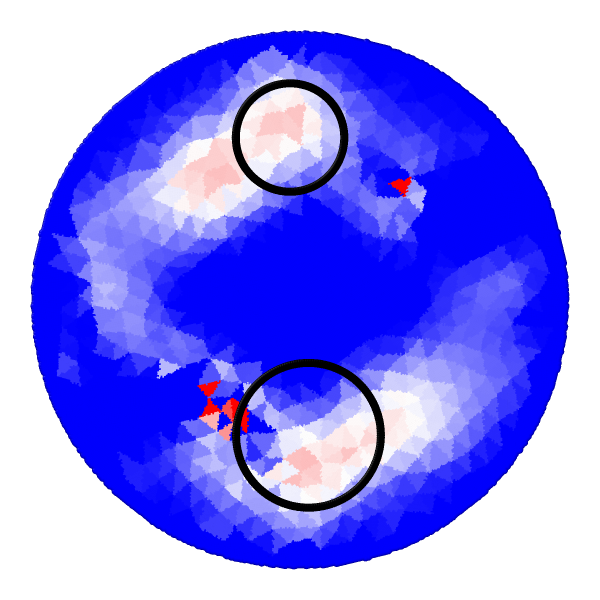}&
\includegraphics[width = \figlen, trim = {0.3cm 0.1cm 0.3cm 0.1cm}, clip]{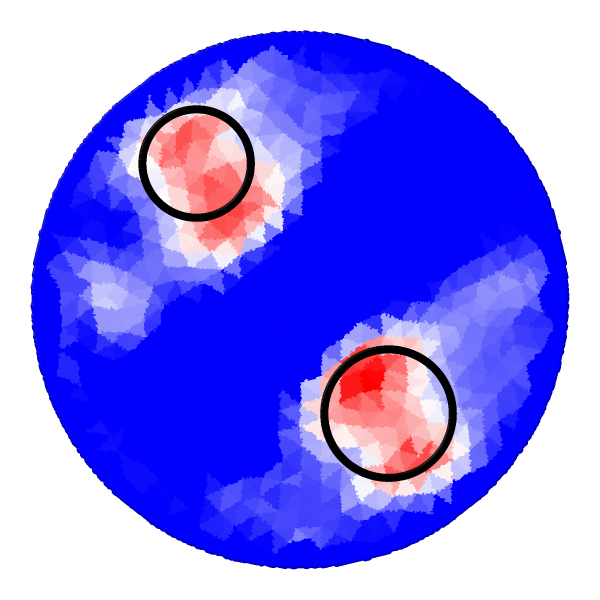}&
\includegraphics[width = \figlen, trim = {0.3cm 0.1cm 0.3cm 0.1cm}, clip]{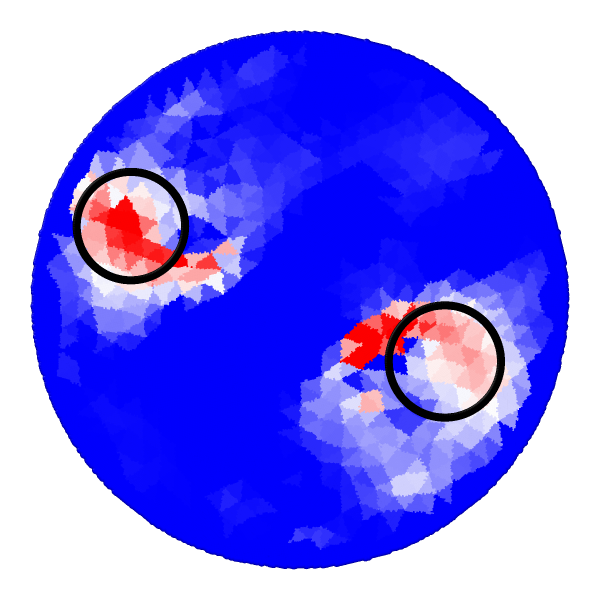}&
\includegraphics[width = \figlen, trim = {0.3cm 0.1cm 0.3cm 0.1cm}, clip]{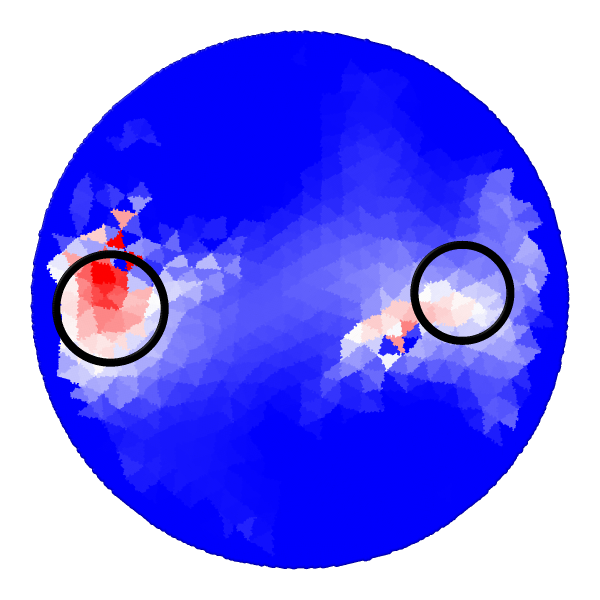}&
\includegraphics[width = \figlen, trim = {0.3cm 0.1cm 0.3cm 0.1cm}, clip]{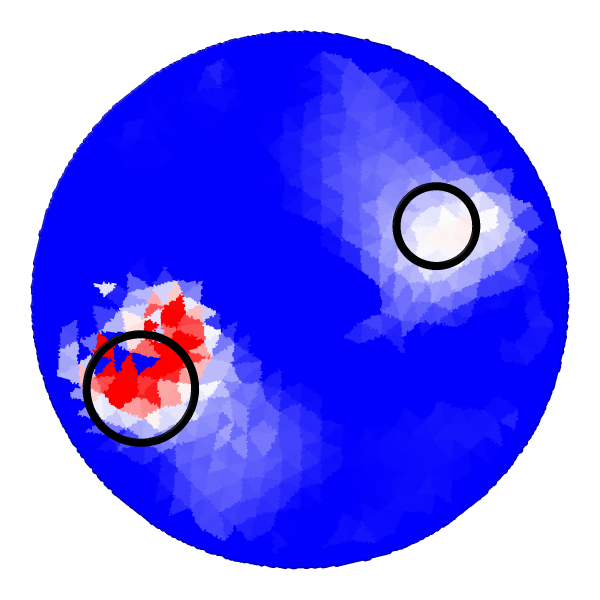}\\
$t=1$ & $t=2$ & $t=3$ & $t=4$ & $t=5$\\
\includegraphics[width = \figlen, trim = {0.3cm 0.1cm 0.3cm 0.1cm}, clip]{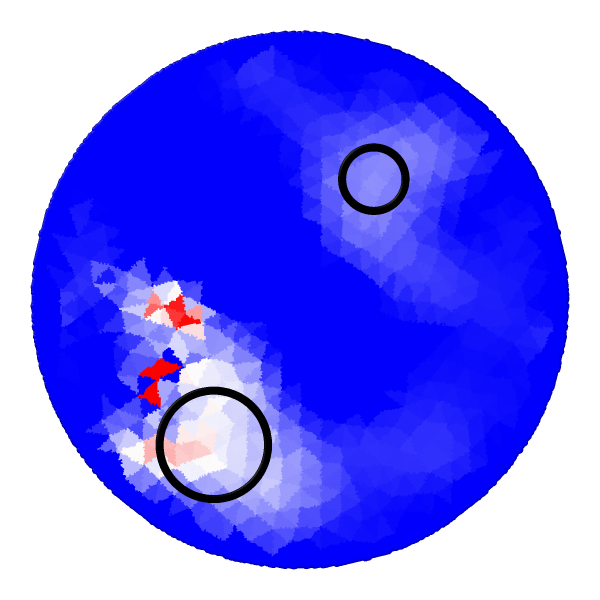}&
\includegraphics[width = \figlen, trim = {0.3cm 0.1cm 0.3cm 0.1cm}, clip]{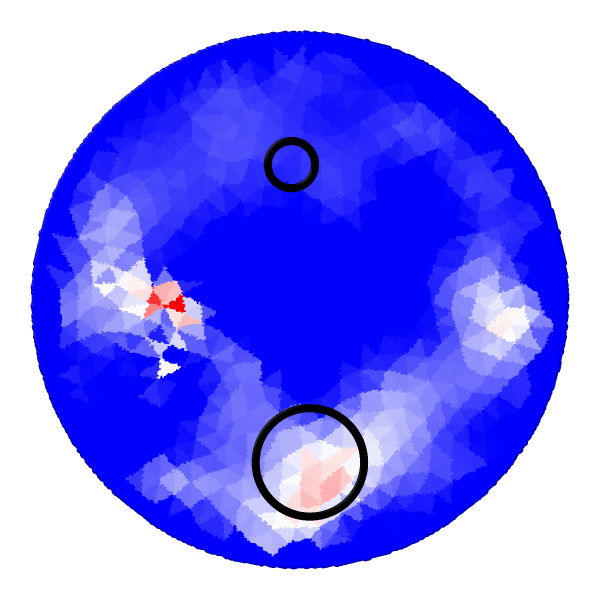}&
\includegraphics[width = \figlen, trim = {0.3cm 0.1cm 0.3cm 0.1cm}, clip]{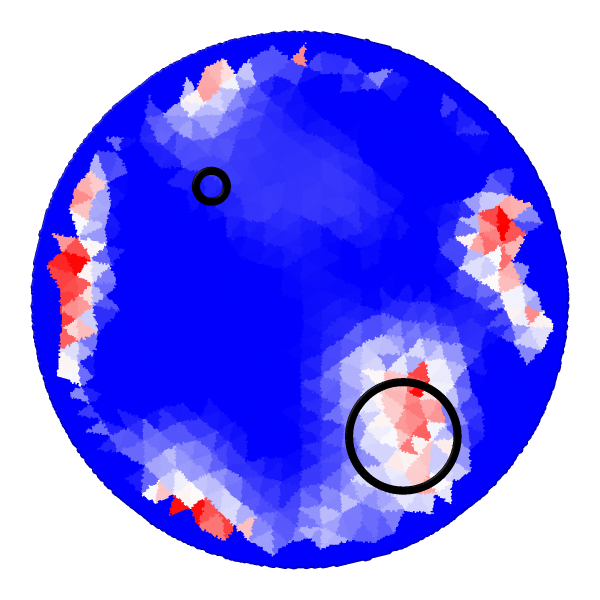}&
\includegraphics[width = \figlen, trim = {0.3cm 0.1cm 0.3cm 0.1cm}, clip]{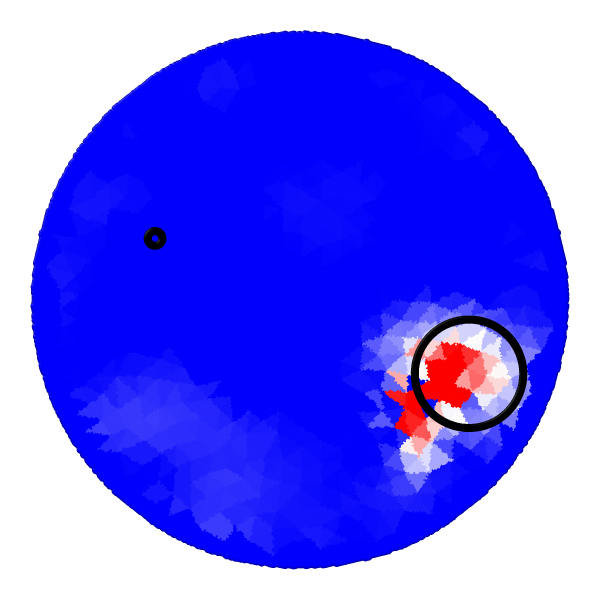}&
\includegraphics[width = \figlen, trim = {0.3cm 0.1cm 0.3cm 0.1cm}, clip]{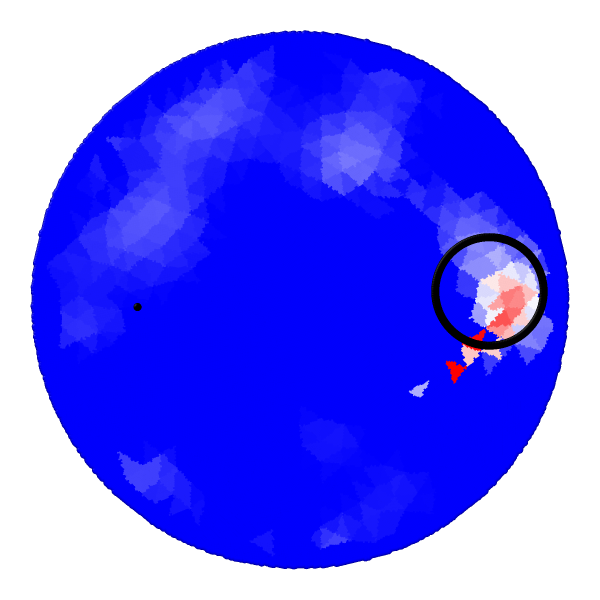}\\
$t=6$ & $t=7$ & $t=8$ & $t=9$ & $t=10$
\end{tabular}
\caption{
Numerical results for Example \ref{exam:5}, which involves diminishing inhomogeneities.\label{fig5}
}
\end{figure}

\section{Conclusion}\label{sec:CON}
In this study we have developed an iterative direct sampling method for reconstruction moving inhomogeneities in parabolic inverse problems.
It incorporates dynamic low-rank adaptation and strategies for improving the robustness of the formulation and can reliably image moving inclusions from noisy data, using only one pair of lateral Cauchy data.
The abstract framework ensures that it is applicable to diverse parabolic inverse problems, including nonlinear and mixed-type scenarios.
Numerical experiments validate the performance of the method across challenging scenarios, e.g., moving, fading, or shape-evolving inclusions, and resolving complex interactions like merging and splitting.
The results confirm the robustness and versatility of the IDSM.
The efficiency and stability are impressive: it requires only 4-5 total PDE solves per Cauchy data pair under the noise level $\varepsilon = 5\%$.
Overall, it is a powerful tool for solving inverse problems in parabolic systems involving evolving inclusions.
Future work may explore further enhancements, e.g., integrating machine learning techniques or extending the method to other time-dependent PDE models.

\bibliographystyle{siam}
\bibliography{ref}
\end{document}